\documentclass[12pt]{amsart}
\usepackage[utf8]{inputenc}

\usepackage[boxsize=1em]{ytableau}
\usepackage{amsmath,amssymb,amsthm,mathtools,mathdots,nicefrac,tikz-cd,bbding,stmaryrd,float,multicol,caption,bbm, upgreek}
\usepackage[mathscr]{euscript}
\usepackage[margin=1in, marginparwidth=1.75cm]{geometry}
\usepackage{color}
\usepackage[all]{xy}

\usepackage[alphabetic]{amsrefs}
\usepackage{hyperref}
\hypersetup{colorlinks=true,citecolor=purple,linkcolor=purple}

\theoremstyle{definition}
\newtheorem{theorem}{Theorem}[section]

\theoremstyle{definition}
\newtheorem{corollary}[theorem]{Corollary}

\theoremstyle{definition}
\newtheorem{proposition}[theorem]{Proposition}

\theoremstyle{definition}
\newtheorem{lemma}[theorem]{Lemma}

\theoremstyle{definition}
\newtheorem{definition}[theorem]{Definition}

\theoremstyle{definition}
\newtheorem{example}[theorem]{Example}

\theoremstyle{definition}

\theoremstyle{definition}

\theoremstyle{definition}
\newtheorem{conjecture}[theorem]{Conjecture}

\theoremstyle{definition}

\newcommand{\bk}{\mathbf{k}}
\newcommand{\Sym}{\text{Sym}}

\newcommand{\Z}{\mathbf{Z}}

\newcommand{\GL}{\mathbf{GL}}

\newcommand{\coker}{\text{coker}\,}

\newcommand{\bfS}{\textbf{S}}

\renewcommand{\char}{\text{char}}

\DeclareMathOperator{\Tab}{Tab}

\newcommand{\interior}[1]{%
  {\kern0pt#1}^{\mathrm{o}}%
}

\title{Standard monomial theory modulo Frobenius in characteristic two}
\author{Laura Casabella}
\address{Max Planck Institute for Mathematics in the Sciences, Leipzig, Germany}
\email{\href{mailto:laura.casabella@mis.mpg.de}{laura.casabella@mis.mpg.de}}
 \urladdr{\url{https://sites.google.com/view/lauracasabella}}

\author{Teresa Yu}
\address{Department of Mathematics, University of Michigan, Ann Arbor, MI}
\email{\href{mailto:twyu@umich.edu}{twyu@umich.edu}}
\urladdr{\url{https://teresaxyu.github.io}}
\thanks{TY was supported by NSF grant DGE-1841052.}

\begin{document}
\maketitle

\begin{abstract}
    Over a field of characteristic two, we develop a theory of standard monomials for polynomial rings modulo a Frobenius power of the maximal ideal generated by all variables. As a result, we obtain a filtration by modular $\GL_n$-representations whose characters are given by particular truncated Schur polynomials, thus proving a conjecture by Gao--Raicu--VandeBogert in the characteristic two case. 
\end{abstract}

\section{Introduction}

There has been recent interest in studying positive characteristic analogues of classical results on flag varieties in characteristic zero, such as the Borel--Weil--Bott theorem on the cohomology of line bundles on the complete flag variety (see \cite{GRV23} for a survey). The study of flag varieties is connected to commutative algebra, representation theory, and combinatorics through standard monomial theory and determinantal ideals \cite{BCRV}. In this paper, we provide results towards developing an analogue of standard monomial theory in positive characteristic by considering determinantal ideals ``modulo Frobenius". 

Let $\bk$ be an algebraically closed field (for now, of arbitrary characteristic), and let $R=\bk[x_1,\ldots,x_n,y_1,\ldots,y_n]$ with bigrading $\deg(x_i)=(1,0),\deg(y_i)=(0,1)$. Let $I\subset R$ be the bihomogeneous ideal generated by the $2\times 2$ minors of the generic matrix
\[\begin{pmatrix} x_1 & x_2 & \cdots & x_n \\ y_1 & y_2 & \cdots & y_n\end{pmatrix}.\]
Then, $R$ has a $\bk$-basis consisting of products of minors of the matrix, and this basis is indexed by semistandard Young tableaux (SSYT) with entries in $[n]=\{1,\ldots,n\}$; the basis elements are called \emph{standard monomials}. In particular, for $a\geq b\geq d$, the $\bk$-vector space $(I^d/I^{d+1})_{(a,b)}$ has a basis indexed by semistandard Young tableaux of shape $(a+b-d,d)$. This basis and its properties are useful for studying many aspects of $I$ and $R/I$ \cite[Ch. 3]{BCRV}. For example, this $\bk$-vector space has an action of $\GL_n$, and since standard monomials are weight vectors, the character of this representation is
\begin{equation}\label{eqn:character char0}
    \left[(I^d/I^{d+1)})_{(a,b)}\right]=s_{(a+b-d,d)},
\end{equation}
where $s_{(a+b-d,d)}$ is the Schur polynomial in $n$ variables indexed by the partition $(a+b-d,d)$.

When $\char(\bk)=p>0$, one can work modulo a Frobenius power of the ideal generated by the variables of $R$, and consider the quotient ring
\[\overline{R}=R/( x_1^p,\ldots,x_n^p,y_1^p,\ldots,y_n^p),\]
and the ideal $\overline{I}=I\overline{R}$. In \cites{DW92,wal94}, the authors give modular analogues of Schur polynomials. The \emph{$p$-truncated complete symmetric polynomials} $h_d^{(p)}\in\Z[t_1,\ldots,t_n]$ are defined by
\[h_d^{(p)}=\sum_{\substack{i_1+\cdots+i_n=d\\ 0\leq i_j<p}} t_1^{i_1}\cdots t_n^{i_n},\]
and the \emph{$p$-truncated Schur polynomials} $s^{(p)}_{(a,b)}\in\Z[t_1,\ldots,t_n]$ are defined by
\[s^{(p)}_{(a,b)}=\det\begin{pmatrix}
    h^{(p)}_a & h^{(p)}_{a+1}\\ h^{(p)}_{b-1} & h^{(p)}_b
    \end{pmatrix}.\]

It is natural to ask whether the analogue of (\ref{eqn:character char0}) holds for $p$-truncated Schur polynomials. Although one can find immediate counterexamples (see for instance \cite[Section 6]{GRV23}), in \cite[Conjecture 6.1]{GRV23} the authors conjecture that the analogue does hold for $a-b \geq p-1$. In this paper, we give a complete answer to the question in the $p=2$ case by finding the character of $(\overline{I}^d/\overline{I}^{d+1})_{(a,b)}$  for all $a\geq b\geq d$.

\begin{theorem}[Theorem~\ref{thm:main}]
    Let $p=2$, and $a\geq b\geq d$. Then,
    \[\left[\left(\overline{I}^d/\overline{I}^{d+1}\right)_{(a,b)}\right]=\begin{cases}
    \sum_{j=1}^a(-1)^{j-1}s^{(2)}_{(a+j,a-j)}&\quad\text{if $a=b=d$},\\ \\
    s^{(2)}_{(a,a)}+\sum_{j=2}^a (-1)^j s^{(2)}_{(a+j,a-j)}&\quad\text{if $a-1=b-1=d$},\\ \\
    s^{(2)}_{(a+b-d,d)}&\quad\text{otherwise.}
    \end{cases}\]
\end{theorem}

Our overall proof strategy is to find a basis indexed by a class of tableaux introduced in \cite{GRV23}, called \emph{$2$-semistandard Young tableaux}; such tableaux of shape $(a,b)$ index the monomials appearing in $s^{(2)}_{(a,b)}$. This basis will be obtained by taking the images of particular classical standard monomials in $\overline{R}$, and this ensures that the elements are in $\overline{I}^d$ for an appropriate $d$. We then use tableaux combinatorics to show that the images of all classical standard monomials are in the span of the particular subset taken to form our basis, and this will be sufficient to obtain the character formulas.

\subsection{Modular representation theory of $\GL_n$}

If $\char(\bk)=0$, then the representation theory of $\GL_n$ over $\bk$ is very well-understood: every simple polynomial representation is given by $\bfS_\lambda(\bk^n)$, where $\bfS_\lambda(-)$ is the Schur functor corresponding to a partition $\lambda$ with at most $n$ rows. The highest weight of this simple module is $\lambda$, and its character is the Schur polynomial $s_\lambda$ in $n$ variables.

If $\char(\bk)=p>0$, then representations of $\GL_n$ are much more difficult to understand. Let $F^p\subset\Sym^p$ denote the \emph{Frobenius power functor}. Then the \emph{truncated symmetric power functor} is given by
\[T_p\,\Sym^d(\bk^n)=\coker\left(F^p(\bk^n)\otimes\Sym^{d-p}(\bk^n)\to\Sym^d(\bk^n)\right).\]
Doty--Walker conjectured that the characters of all simple $\GL_n$-representations could be understood via tensor products of representations of the form $T_p\,\Sym^d(\bk^n)$, leading to the study of $p$-truncated Schur polynomials \cites{DW92, wal94}. More recently, generalizations of the $h^{(p)}_d$ polynomials have been studied in the context of symmetric function theory \cites{FM, grinberg}. The $p$-truncated Schur polynomials are virtual $\GL_n$-characters, and \cite[Conjecture 6.1]{GRV23} predicts a particular $\GL_n$-representation realizing some of them as actual characters.

As $\GL_n$-representations, we have that
\[\overline{R}_{(a,b)}=T_p\,\Sym^a(\bk^n)\otimes T_p\,\Sym^b(\bk^n)\]
with character $h^{(p)}_a\cdot h^{(p)}_b$. Our main result shows that for $p=2$ and $a>b$, $\overline{R}_{(a,b)}$ has a filtration by $\GL_n$-representations with characters $s^{(2)}_{(a+b-d,d)}$ for $0\leq d\leq b$.

A related open question concerns which simple $\GL_n$-representations can be realized as a composition factor of tensor products of truncated symmetric powers. As in the characteristic zero case, simple modules are indexed by partitions $\lambda$ corresponding to the highest weight of the module. A partition $\lambda$ is called \emph{$p$-restricted} if $\lambda_i-\lambda_{i-1}\leq p-1$ for all $i$.

\begin{conjecture}[Conjecture 6.4, \cite{GRV23}]\label{conj:grv simples}
    If $a-b\geq p-1$, then every composition factor $L(\lambda)$ of $T_p\,\Sym^a(\bk^n)\otimes T_p\,\Sym^b(\bk^n)$ is $p$-restricted.
\end{conjecture}

For $p=2$, this conjecture is a straightforward consequence of Pieri's rule; it is provided in \S \ref{subsec:irreps dir}.

\subsection{Outline} The rest of the paper is organized as follows. In \S\ref{sec:prelim}, we provide definitions and properties on $2$-semistandard Young tableaux. In \S\ref{sec:basis construction}, we construct elements of $\overline{I}^d_{(a,b)}$ corresponding to $2$-semistandard Young tableaux, and in \S\ref{sec:span}, we show that this set of elements forms a $\bk$-spanning set of $\overline{R}_{(a,b)}$. In \S\ref{sec:main thm}, we prove our main result on the character and discuss future work related to simple modular $\GL_n$-representations and developing a standard monomial theory over exterior algebras.

\subsection{Notation} We list some of the notation that will be used throughout the remainder of the paper.
\begin{itemize}
    \item $R=\bk[x_1,\ldots,x_n,y_1,\ldots,y_n]$, where $\bk$ is a field of characteristic $2$
    \item $I\subset R$ is the ideal generated by all $2\times 2$ minors of
    \[\begin{pmatrix} x_1 & x_2 & \cdots & x_n \\ y_1 & y_2 & \cdots & y_n\end{pmatrix}\]
    \item $\overline{R}=R/(x_1^2,\ldots,x_n^2,y_1^2,\ldots,y_n^2)$, and $\overline{I}=I\overline{R}$
    \item $\overline{f}\in\overline{R}$ denotes the image of $f\in R$
    \item $\Tab_{(a,b)}$, $\Tab^{(2)}_{(a,b)}$ are the ($2$-)SSYT of shape $(a,b)$
    \item $s_{(a,b)}$, $s^{(2)}_{(a,b)}$ are ($2$-truncated) Schur polynomials
    \item $h_d$, $h_d^{(2)}$ are ($2$-truncated) complete symmetric polynomials
    \item $F_{T,a}$ is the element of $I^d_{(a,b)}$ corresponding to $T$ a tableau of shape $(a+b-d,d)$
    \item $G_{T,a}$ is the element of $\overline{I}^d_{(a,b)}$ corresponding to a $2$-SSYT $T$
    \item $\mathcal{B}_{a,b,d}$ is the set of $2$-standard monomials in $\overline{I}^d_{(a,b)}$
    \item $\mathcal{B}_{a,b}$ is the union of $\mathcal{B}_{a,b,d}$ over $0\leq d\leq b$
    \item $\Phi_{a,b,d}$ is the map from a set of $2$-SSYT to $\Tab_{(a+b-d,d)}$
\end{itemize}

\subsection*{Acknowledgements}
We thank the organizers of the 2023 Pragmatic workshop held at the University of Catania, where this project began. We are especially grateful to Claudiu Raicu for suggesting this problem, and to him and Alessio Sammartano for many helpful discussions and feedback. Computational experiments using \texttt{Macaulay2} \cite{M2} provided valuable insights throughout our work, and we thank Van Vo for assistance with these computations.

\section{Preliminaries}\label{sec:prelim}

We begin by recalling the definitions of semistandard Young tableaux and classical standard monomials.

\begin{definition}
    A Young tableau $T$ of shape $(a,b)$ given by
    \begin{displaymath}
\ytableausetup
{mathmode,boxsize=1.5em}
T=\begin{ytableau}
\alpha_1 & \alpha_2 & \cdots & \alpha_b & \cdots & \alpha_a\\
\beta_1 & \beta_2 & \cdots & \beta_b
\end{ytableau}
\end{displaymath}
   is called \emph{semistandard} if its indices are weakly increasing along rows and strictly increasing along columns:
    \[\alpha_1\leq\cdots\leq \alpha_a,\quad\beta_1\leq\cdots\leq\beta_b,\quad \alpha_i<\beta_i\text{ for all $1\leq i\leq b$.}\]
    We let $\Tab_{(a,b)}$ denote the set of semistandard Young tableaux of shape $(a,b)$ with indices in $[n]$.
\end{definition}

For a tableau $T$ of shape $(a,b)$, let
\[t^T\coloneqq t_{\alpha_1}\cdots t_{\alpha_a}t_{\beta_1}\cdots t_{\beta_b}\in \Z[t_1,\ldots,t_n].\]
Then for $a\geq b\geq d$, we have that
\begin{equation}\label{eqn:schur tab}
s_{(a+b-d,d)}=\sum_{T\in\Tab_{(a+b-d,d)}}t^T.
\end{equation}
Furthermore, we have a basis of $(I^d/I^{d+1})_{(a,b)}$ indexed by $\Tab_{(a+b-d,d)}$. For $i\neq j$ with $1\leq i,j\leq n$, let $[i,j]$ denote the $2$-minor $x_iy_j-x_jy_i\in R$. For a tableau $T$ (not necessarily semistandard) of shape $(a+b-d,d)$, let
\[F_{T,a}=\left(\prod_{i=1}^d [\alpha_i,\beta_i]\right) x_{\alpha_{d+1}}\cdots x_{\alpha_a}y_{\alpha_{a+1}}\cdots y_{\alpha_{a+b-d}}.\]
The residue classes of elements of the set $\{F_{T,a}:T\in \Tab_{(a+b-d,d)}\}$ form a basis of $(I^d/I^{d+1})_{(a,b)}$. One way to show linear independence is by noting that the initial term of $F_{T,a}$ with respect to the reverse lexicographic order on $R$ is
\[x_{\alpha_1}\cdots x_{\alpha_a}y_{\beta_1}\cdots y_{\beta_d}y_{\alpha_{a+1}}\cdots y_{\alpha_{a+b-d}}.\]

The following modification to the definition of semistandard Young tableaux is motivated by obtaining an analogue of (\ref{eqn:schur tab}) for $p$-truncated Schur polynomials \cite[Definition 6.2]{GRV23}.

\begin{definition}
    A tableau $T$ of shape $(a,b)$ is \emph{$p$-semistandard} if the following hold:
    \begin{enumerate}
        \item $T$ is weakly increasing along rows and columns:
        \[\alpha_1\leq\cdots\leq \alpha_a,\quad \beta_1\leq\cdots\leq \beta_b,\quad \alpha_i\leq \beta_i\text{ for $1\leq i\leq b$.}\]
        \item Any constant sequence within each row has length at most $(p-1)$:
        \[\alpha_i<\alpha_{i+p-1}\quad\text{and}\quad \beta_j<\beta_{j+p-1}\text{ for all $i,j$.}\]
        \item If $\alpha_j=\beta_j$, consider the unique indices $r\geq j$ and $s\leq j$ so that
        \[\alpha_j=\cdots=\alpha_r<\alpha_{r+1},\quad \beta_{s-1}<\beta_s=\cdots=\beta_j.\]
        We have $(r-j+1)+(j-s+1)\geq p$, i.e., the total number of entries to the right of (and including) $\alpha_j$ and to the left of (and including) $\beta_j$ is at least $p$.
    \end{enumerate}
\end{definition}

Notice that $T$ is $2$-semistandard if and only if its transpose tableau is semistandard in the classical sense. We now prove the analogue of (\ref{eqn:schur tab}) in the case $p=2$ using this fact. For a partition $\lambda$, let $\lambda^\dagger$ denote the transpose partition.

\begin{proposition}\label{prop:schur poly tab corr}
    Let $a\geq b$. Then,
    \[s^{(2)}_{(a,b)}=\sum_{T\in \Tab^{(2)}_{(a,b)}} t^T.\]
\end{proposition}

\begin{proof}
    Note that $h_d^{(2)}$ is the sum of all degree $d$ squarefree monomials in $t_1,\ldots,t_n$, which is exactly the $d^{\text{th}}$ classical elementary symmetric polynomial $e_d$. We then see that
    \[s^{(2)}_{(a,b)}=\det\begin{pmatrix} h_a^{(2)} & h_{a+1}^{(2)}\\ h_{b-1}^{(2)} & h_b^{(2)}\end{pmatrix}=\det\begin{pmatrix} e_a & e_{a+1} \\ e_{b-1} & e_b\end{pmatrix}=s_{(a,b)^\dagger},\]
    where the last equality follows from the Jacobi--Trudi identity. Thus,
    \[s^{(2)}_{(a,b)}=s_{(a,b)^\dagger}=\sum_{T\in\Tab_{(a,b)^\dagger}}t^T=\sum_{T\in\Tab^{(2)}_{(a,b)}}t^T.\]
\end{proof}

Since $\char(\bk)=2$, the ring $\overline{R}$ is the exterior algebra on $2n$ variables, and we make further remarks on this in \S \ref{subsec:exterior}.

\subsection{$2$-Straightening Rules}

The straightening rules for Young tableaux are essential for showing that classical standard monomials span $R$. They are induced by Pl\"ucker relations, and are independent of the characteristic of $\bk$. In what follows, we specialize to working over a field of characteristic $2$, and deduce further relations among minors in $\overline{R}$. We will use these relations in \S \ref{sec:span} to show that all images of standard monomials are in the span of the proposed basis that we define in \S \ref{sec:basis construction}.

The following result is the key straightening rule for our case of $2$-minors over a field of characteristic $2$ (see \cite[Lemma 3.2.3]{GRV23} for a more general statement).

\begin{proposition}[Classical straightening in characteristic $2$]\label{prop:classical straightening}
Let $\alpha,\beta,\gamma,\delta\in[n]$ be distinct. Then the following relations hold in $R$:
    \begin{align*}
        [\alpha,\beta][\gamma,\delta]+[\alpha,\gamma][\beta,\delta]+[\alpha,\delta][\beta,\gamma]&= 0,\\
        [\alpha,\beta]x_\gamma+[\alpha,\gamma]x_\beta+[\beta,\gamma]x_\alpha&=0.
    \end{align*}
\end{proposition}

We now provide ``$2$-straightening" rules for Young tableaux. These rules give relations among images of classical standard monomials in $\overline{R}$.

\begin{proposition}\label{prop:main 2-straightening rule}
    Let $a \geq 2$ and $m = a-2$. Let $T$ be a tableau of shape $(a,a)$ given by
    \begin{displaymath}
    \ytableausetup
    {mathmode,boxsize=2em}
    \begin{ytableau}
    \alpha & \alpha & \beta_1 & \beta_2 & \cdots & \cdots & \beta_{m-1} &  \beta_m \\
    \beta_1 & \beta_2 & \cdots & \cdots & \cdots & \beta_m & \delta & \varepsilon 
    \end{ytableau}
    \end{displaymath}
    Then in $\overline{R}$,
    \[\overline{F}_{T,a}=\left(x_\alpha y_\alpha\prod_{i=1}^m x_{\beta_i}y_{\beta_i} \right)[\delta,\varepsilon].\]
\end{proposition}

\begin{proof}
    We proceed by induction on $m$. For $m=0$, we have that
    \[[\alpha,\delta][\alpha,\varepsilon]=x_\alpha x_\varepsilon y_\alpha y_\delta + x_\alpha x_\delta y_\alpha y_\varepsilon=x_\alpha y_\alpha [\delta,\varepsilon].\]
    Now assume $m>0$. Assuming the statement has been proven for $m-1$, we have that
    \begin{align*}
        \overline{F}_{T,a}&=\left(x_\alpha y_\alpha\prod_{i=1}^{m-1} x_{\beta_i}y_{\beta_i} \right)[\beta_m,\delta][\beta_m,\varepsilon]\\
        &=\left(x_\alpha y_\alpha\prod_{i=1}^{m-1} x_{\beta_i}y_{\beta_i} \right)x_{\beta_m}y_{\beta_m}[\delta,\varepsilon].
    \end{align*}
\end{proof}

\begin{example}
    Let $a=b=d=4$, and let $T$ be a tableau of shape $(a,a)$ given by
        \begin{displaymath}
\ytableausetup
{mathmode,boxsize=1.5em}
\begin{ytableau}
1 & 1 & 2 & 3 \\
2 & 3 & 4 & 5  
\end{ytableau}
\end{displaymath}
Then, in $\overline{R}$, we have
\begin{align*}
 \overline{F}_{T,4} &= (x_1y_2+x_2y_1)(x_1y_3+x_3y_1)(x_4y_2+x_2y_4)(x_3y_5+x_5y_3)\\
 &= x_1y_1(x_2y_2)(x_3y_3)(x_4y_5+x_5y_4).
\end{align*}
\end{example}

\begin{corollary}\label{cor:2-straightening swap 3,3}
    Let $S,T$ be tableaux of shape $(3,3)$ given by
    \begin{displaymath}
    \ytableausetup
    {mathmode,boxsize=2em}
    S=\begin{ytableau}
    \alpha & \alpha & \gamma \\
    \delta & \varepsilon & \eta
    \end{ytableau}\,\, ,
    \quad T = \begin{ytableau}
    \varepsilon & \eta & \gamma \\
    \delta & \alpha & \alpha
    \end{ytableau}
    \end{displaymath}
    Then $\overline{F}_{S,3}=\overline{F}_{T,3}$ in $\overline{R}$.
\end{corollary}

\begin{proof}
    By Proposition~\ref{prop:main 2-straightening rule}, 
    \[\overline{F}_{S,3}=x_\alpha y_\alpha[\delta,\epsilon][\gamma,\eta]=\overline{F}_{T,3}.\]
\end{proof}

\begin{corollary}\label{cor:2-straightening pairs}
    Let $m\geq 0$, and let $S,T,U$ be tableaux of shape $(a,a)$ given by
    \begin{displaymath}
    \ytableausetup
    {mathmode,boxsize=2em}
    S=\begin{ytableau}
    \alpha & \alpha & \beta_1 & \beta_2 & \cdots & \cdots & \beta_{m-1} &  \beta_m & \gamma \\
    \beta_1 & \beta_2 & \cdots & \cdots & \cdots & \beta_m & \delta & \varepsilon & \eta
    \end{ytableau}
    \end{displaymath}
    
    \begin{displaymath}
    \ytableausetup
    {mathmode,boxsize=2em}
    T = \begin{ytableau}
    \alpha & \alpha & \beta_1 & \beta_2 & \cdots & \cdots & \beta_{m-1} &  \beta_m & \delta \\
    \beta_1 & \beta_2 & \cdots & \cdots & \cdots & \beta_m & \gamma & \varepsilon & \eta
    \end{ytableau}
    \end{displaymath}

    \begin{displaymath}
    \ytableausetup
    {mathmode,boxsize=2em}
    U = \begin{ytableau}
    \alpha & \alpha & \beta_1 & \beta_2 & \cdots & \cdots & \beta_{m-1} &  \beta_m & \varepsilon \\
    \beta_1 & \beta_2 & \cdots & \cdots & \cdots & \beta_m & \gamma & \delta & \eta
    \end{ytableau}
    \end{displaymath}

    Then $\overline{F}_{S,a}+\overline{F}_{T,a}+\overline{F}_{U,a}=0$ in $\overline{R}$.

    The same statement holds when $S,T,U$ are of shape $(a,a-1)$ and the box with index $\eta$ is removed from each.
\end{corollary}

\begin{proof}
    By Proposition~\ref{prop:main 2-straightening rule},
    \begin{align*}
        \overline{F}_{S,a}+\overline{F}_{T,a}+\overline{F}_{U,a}&=x_\alpha y_\alpha\left(\prod_{i=1}^m x_{\beta_i}y_{\beta_i}\right)([\delta,\varepsilon][\gamma,\eta]+[\gamma,\varepsilon][\delta,\eta]+[\gamma,\delta][\varepsilon,\eta]).
    \end{align*}
    This is equal to $0$ in $\overline{R}$ by Proposition~\ref{prop:classical straightening}. The same argument holds if the boxes with index $\eta$ are removed from each tableau.
\end{proof}

\begin{example}
Let 

    \begin{displaymath}
    \ytableausetup
    {mathmode,boxsize=2em}
    S=\begin{ytableau}
    1 & 1 & 2 & 4 & 5 \\
    2 & 4 & 6 & 7 & 8
    \end{ytableau}
    \end{displaymath}
    
    \begin{displaymath}
    \ytableausetup
    {mathmode,boxsize=2em}
    T=\begin{ytableau}
    1 & 1 & 2 & 4 & 6 \\
    2 & 4 & 5 & 7 & 8
    \end{ytableau}
    \end{displaymath}

    \begin{displaymath}
    \ytableausetup
    {mathmode,boxsize=2em}
    U=\begin{ytableau}
    1 & 1 & 2 & 4 & 7 \\
    2 & 4 & 5 & 6 & 8
    \end{ytableau}
    \end{displaymath}
By Proposition ~\ref{prop:main 2-straightening rule} or direct computation,

\begin{align*}
 \overline{F}_{S,5} &= x_1x_2x_4y_1y_2y_4(x_7y_6+x_6y_7)(x_8y_5+x_5y_8),\\
 \overline{F}_{T,5} &= x_1x_2x_4y_1y_2y_4(x_7y_5+x_5y_7)(x_8y_6+x_6y_8),\\
 \overline{F}_{U,5} &= x_1x_2x_4y_1y_2y_4(x_6y_5+x_5y_6)(x_8y_7+x_7y_8).\\
\end{align*}
Then $\overline{F}_{S,5} + \overline{F}_{T,5} + \overline{F}_{U,5} = 0$ in $\overline{R}$.
\end{example}

\begin{corollary}\label{cor:2-straightening swap 3,2}
    Let $S,T$ be tableaux of shape $(3,2)$ given by
    \begin{displaymath}
    \ytableausetup
    {mathmode,boxsize=2em}
    S=\begin{ytableau}
    \alpha & \alpha & \gamma \\
    \delta & \varepsilon 
    \end{ytableau}\,\, ,
    \quad T = \begin{ytableau}
    \delta & \alpha & \alpha \\
    \varepsilon & \gamma 
    \end{ytableau}
    \end{displaymath}
    Then $\overline{F}_{S,3}=\overline{F}_{T,3}$ in $\overline{R}$.
\end{corollary}

\begin{proof}
    Expanding $\overline{F}_{T,3}$, we obtain
    \[\overline{F}_{T,3}=[\delta,\varepsilon][\alpha,\gamma]x_\alpha=[\delta,\varepsilon]x_\alpha y_\alpha x_\gamma=\overline{F}_{S,3},\]
    where the last equality follows from Proposition~\ref{prop:main 2-straightening rule}.
\end{proof}

\section{Construction of 2-standard monomials}\label{sec:basis construction}

For $a\geq b\geq d$, we construct the set $\mathcal{B}_{a,b,d}\subset\overline{I}^d_{(a,b)}$ of elements of $\overline{R}$; the residue classes of these elements will be shown to be a $\bk$-basis for $(\overline{I}^d/\overline{I}^{d+1})_{(a,b)}$. We consider the following cases:

\begin{enumerate}
    \item neither $a=b=d$ nor $a-1=b-1=d$ hold (we call this the \emph{general case});
    \item $a=b=d$;
    \item $a-1=b-1=d$.
\end{enumerate}

In case (1), we will see that the basis is indexed by $\Tab^{(2)}_{(a+b-d,d)}$. In case (2), the basis will be indexed by the subset of $\Tab^{(2)}_{(a,a)}$ consisting of tableaux with non-identical rows. In case (3), the basis will be indexed by the union of $\Tab^{(2)}_{(a+1,a-1)}$ and the subset of $\Tab^{(2)}_{(a,a)}$ consisting of tableaux with identical rows.

In each case, we begin with a $2$-semistandard tableau $T$, and we modify it to obtain a semistandard one $\widetilde{T}$. The basis element $G_{T,a}$ corresponding to $T$ will be given by $\overline{F}_{\widetilde{T},a}$.

\subsection{Case (1): general case}
Suppose $a\geq b\geq d$ are such that neither $a=b=d$ nor $a-1=b-1=d$ hold, and let $T\in \Tab^{(2)}_{(a+b-d,d)}$.
Note that $a+b-d>d$. We construct a tableau $\widetilde{T}\in \Tab_{(a+b-d,d)}$ as follows.

Suppose $T$ has a subtableau of the form
\begin{displaymath}
\ytableausetup
{mathmode,boxsize=2em}
\begin{ytableau}
\alpha_i & \alpha_{i+1} & \cdots & \alpha_k \\
\alpha_i & \alpha_{i+1} & \cdots & \alpha_k
\end{ytableau}
\end{displaymath}
that is not strictly contained in another such subdiagram. 

If $k<d$, then the $i,\ldots,k+1$ columns of $T$ are given by
\begin{displaymath}
\ytableausetup
{mathmode,boxsize=2em}
\begin{ytableau}
\alpha_i & \alpha_{i+1} & \alpha_{i+2} & \cdots & \alpha_k & \alpha_{k+1} \\
\alpha_i & \alpha_{i+1} & \alpha_{i+2} & \cdots & \alpha_k & \beta_{k+1}
\end{ytableau}
\end{displaymath}
where $\alpha_{k+1}<\beta_{k+1}$. Let the $i,\ldots,k+1$ columns of $\widetilde{T}$ be given by
\begin{displaymath}
\ytableausetup
{mathmode,boxsize=2em}
\begin{ytableau}
\alpha_i & \alpha_i & \alpha_{i+1} & \alpha_{i+2} & \cdots & \alpha_{k-1} & \alpha_k \\
\alpha_{i+1} & \alpha_{i+2} & \alpha_{i+3} & \cdots & \alpha_k & \alpha_{k+1} & \beta_{k+1}
\end{ytableau}
\end{displaymath}
In words, the sequences $\alpha_{i+1},\ldots \alpha_k,\alpha_{k+1}$ in the first row of $T$ and $\alpha_i,\alpha_{i+1},\ldots,\alpha_k$ in the second row of $T$ are swapped.

Similarly, if $k=d$, then the $i,\ldots,d+1$ columns of $\widetilde{T}$ are given by
\begin{displaymath}
\ytableausetup
{mathmode,boxsize=2em}
\begin{ytableau}
\alpha_i & \alpha_i & \alpha_{i+1} & \alpha_{i+2} & \cdots & \alpha_{d-1} & \alpha_d \\
\alpha_{i+1} & \alpha_{i+2} & \alpha_{i+3} & \cdots & \alpha_d & \alpha_{d+1} 
\end{ytableau}
\end{displaymath}
Note in this case that $\alpha_{d+1}$ exists since $a+b-d>d$.

Do this procedure whenever a column in $T$ has repeated entries to fill in the corresponding entries of $\widetilde{T}$; all other entries of $\widetilde{T}$ are the same as those in $T$. We then have the following:
\begin{itemize}
    \item $\widetilde{T}$ is well-defined. If there are repeated indices in a column, then we can always eventually find a column after it with different indices in the column to apply the first case of swapping above, or we reach the box $\alpha_{d+1}$, and we apply the second case of swapping above. This does not introduce any columns with repeated indices. In addition, if $T$ has two subdiagrams of the above forms, then performing the swap for one subdiagram does not affect columns of the other subdiagram, and so the order in which we perform the swaps does not matter.
    \item $\widetilde{T}$ has shape $(a+b-d,d)$.
    \item $\widetilde{T}$ is semistandard. In the columns of $\widetilde{T}$ affected by swaps, columns are now strictly increasing, since $\alpha_i<\alpha_{i+1}<\cdots < \alpha_k$, and $\alpha_k<\alpha_{k+1}<\beta_{k+1}$. Other columns remain the same, and so all columns of $\widetilde{T}$ are strictly increasing. Rows are weakly increasing, as swapping does not introduce strict decreases.
\end{itemize}

Let $F_{\widetilde{T},a}\in I^d_{(a,b)}\subset R$. Then define $G_{T,a}$ to be
\[G_{T,a}\coloneqq\overline{F}_{\widetilde{T},a}\in\overline{I}^d_{(a,b)}.\]
The construction above defines a map

\begin{align*}
    \Phi_{a,b,d}:  \Tab^{(2)}_{(a+b-d,d)} &\to \Tab_{(a+b-d,d)} \\
    T &\mapsto \widetilde{T}.
\end{align*}

\begin{example}
Let $a=5, b=d=4$. If $T$ is given by
        \begin{displaymath}
\ytableausetup
{mathmode,boxsize=1.5em}
\begin{ytableau}
1 & 2 & 3 & 5 & 6 \\
1 & 2 & 4 & 5  
\end{ytableau}
\end{displaymath}
then its image $\widetilde{T}$ is given by
        \begin{displaymath}
\ytableausetup
{mathmode,boxsize=1.5em}
\begin{ytableau}
1 & 1 & 2 & 5 & 5 \\
2 & 3 & 4 & 6  
\end{ytableau}
\end{displaymath}
and thus, in $\overline{R}$,
\begin{align*}
G_{T,5} &= (x_1y_2 + x_2y_1)(x_1y_3+x_3y_1)(x_2y_4+x_4y_2)(x_5y_6+x_6y_5)x_5 \\
&= x_1x_2x_5x_6y_1y_2y_5(x_3y_4+x_4y_3).
\end{align*}
\end{example}

\subsection{Case (2): $a=b=d$}
We now consider the special case $a=b=d$. Here, the basis $\mathcal{B}_{a,a,a}$ will be indexed by a subset of $\Tab^{(2)}_{(a,a)}$.

Let $T\in \Tab^{(2)}_{(a,a)}$ be such that the two rows are not identical. For sequences of identical columns not involving the last column (the last column does not have a repeated index), then do the swaps as above in the general case. 

Now consider sequences of identical columns involving the last column, so $\alpha_a=\beta_a$. Then $T$ contains
\begin{displaymath}
\ytableausetup
{mathmode,boxsize=2em}
\begin{ytableau}
\alpha_{i-1} & \alpha_i & \alpha_{i+1} & \cdots & \alpha_{a-1} & \alpha_a \\
\beta_{i-1} & \alpha_{i} & \alpha_{i+1} & \cdots & \alpha_{a-1} & \alpha_{a} 
\end{ytableau}
\end{displaymath}
where $\alpha_{i-1}<\beta_{i-1}$. Note that such an $(i-1)^\text{st}$ column exists, since the two rows of $T$ are not identical. Let the $i-1,\ldots,a$ columns of $\widetilde{T}$ be given by
\begin{displaymath}
\ytableausetup
{mathmode,boxsize=2em}
\begin{ytableau}
\alpha_{i-1} & \beta_{i-1} & \alpha_{i} & \cdots & \alpha_{a-2} & \alpha_{a-1} \\
\alpha_{i} & \alpha_{i+1} & \alpha_{i+2} & \cdots & \alpha_{a} & \alpha_{a} 
\end{ytableau}
\end{displaymath}
In words, the sequences $\alpha_i,\ldots,\alpha_a$ of the first row and $\beta_{i-1},\alpha_i,\ldots,\alpha_{a-1}$ of the second row are swapped. Note that this subtableau of $\widetilde{T}$ is now semistandard.

If $T$ has repeats in columns right before the $(i-1)^\text{st}$ column, then the swapping processes commute: the first kind of swap only involves the box in row $1$, column $(i-1)$ and boxes in earlier columns, while the second kind of swap only involves the box in row $2$, column $(i-1)$ and boxes in later columns.

In this case, $\widetilde{T}$ is still semistandard. If the columns $i-2,i-1,i$ of $T$ are 
        \begin{displaymath}
\ytableausetup
{mathmode,boxsize=2em}
\begin{ytableau}
\alpha_{i-2} & \alpha_{i-1} & \alpha_i \\
\alpha_{i-2} & \beta_{i-1} & \alpha_i 
\end{ytableau}
\end{displaymath}
then these columns of $\widetilde{T}$ are
        \begin{displaymath}
\ytableausetup
{mathmode,boxsize=2em}
\begin{ytableau}
\alpha_{i-2} & \alpha_{i-2} & \beta_{i-1} \\
\alpha_{i-1} & \alpha_i & \alpha_i 
\end{ytableau}
\end{displaymath}
which is semistandard.

Let
\[G_{T,a}\coloneqq\overline{F}_{\widetilde{T},a}\in\overline{I}^a_{(a,a)}.\]
Let $\mathcal{D}\subset\Tab^{(2)}_{(a,a)}$ denote the set of $2$-semistandard Young tableaux of shape $(a,a)$ that do not have repeated rows. Then the above construction defines a map
\[\Phi_{a,a,a}:\mathcal{D}\to\Tab_{(a,a)}.\]

\begin{example}
    If $T$ is given by
        \begin{displaymath}
\ytableausetup
{mathmode,boxsize=1.5em}
\begin{ytableau}
1 & 2 & 3 & 5 & 6 \\
1 & 2 & 4 & 5 & 6
\end{ytableau}
\end{displaymath}
then its image $\widetilde{T}$ is given by
        \begin{displaymath}
\ytableausetup
{mathmode,boxsize=1.5em}
\begin{ytableau}
1 & 1 & 2 & 4 & 5 \\
2 & 3 & 5 & 6 & 6
\end{ytableau}
\end{displaymath}
\end{example}

\subsection{Case (3): $a-1=b-1=d$}
In this case, $(a+b-d,d)=(a+1,a-1)$. First suppose $T\in \Tab^{(2)}_{(a+1, a-1)}$. Then perform the same construction as in the $a>b>d$ case to obtain $\widetilde{T}$ and $G_{T,a}$.

Now suppose $T\in \Tab^{(2)}_{(a,a)}$ and is of the form
\begin{displaymath}
\ytableausetup
{mathmode,boxsize=1.5em}
\begin{ytableau}
\alpha_1 & \alpha_2 & \cdots & \alpha_a \\
\alpha_1 & \alpha_2 & \cdots & \alpha_a
\end{ytableau}
\end{displaymath}
Then
\begin{displaymath}
\ytableausetup
{mathmode,boxsize=2em}
\begin{ytableau}
\alpha_1 & \alpha_2 & \cdots & \alpha_{a-1} & \alpha_a & \alpha_a \\
\alpha_1 & \alpha_2 & \cdots & \alpha_{a-1}
\end{ytableau}
\end{displaymath}
is a tableau of shape $(a+1,a-1)$ with the same weight as $T$. Note that it is not $2$-semistandard. Apply the swapping process to obtain $\widetilde{T}$
\begin{displaymath}
\ytableausetup
{mathmode,boxsize=2em}
\begin{ytableau}
\alpha_1 & \alpha_1 & \cdots & \alpha_{a-2} & \alpha_{a-1} & \alpha_a \\
\alpha_2 & \alpha_3 & \cdots & \alpha_{a}
\end{ytableau}
\end{displaymath}
Then $\widetilde{T}\in \Tab_{(a+1,a-1)}$. The above construction defines a map
\[\Phi_{a,a,a-1}:\Tab^{(2)}_{(a+1,a-1)}\cup\,(\Tab^{(2)}_{(a,a)}\setminus\mathcal{D})\to\Tab_{(a+1,a-1)}.\]
The basis $\mathcal{B}_{a,a,a-1}$ will therefore be indexed by tableaux in $\Tab^{(2)}_{(a,a)}$ with identical rows and $\Tab^{(2)}_{a+1,a-1}$.

\begin{example}
    If $T$ is given by
        \begin{displaymath}
\ytableausetup
{mathmode,boxsize=1.5em}
\begin{ytableau}
1 & 2 & 4 & 5 \\
1 & 2 & 4 & 5  
\end{ytableau}
\end{displaymath}
then its image $\widetilde{T}$ is given by
        \begin{displaymath}
\ytableausetup
{mathmode,boxsize=1.5em}
\begin{ytableau}
1 & 1 & 2 & 4 & 5 \\
2 & 4 & 5  
\end{ytableau}
\end{displaymath}
\end{example}



\subsection{2-straight tableaux}

In the previous subsection, we defined the sets $\mathcal{B}_{a,b,d}\subset\overline{I}^d_{(a,b)}$. They consist of images of classical standard monomials in $\overline{R}$.

\begin{definition}
    The elements of $\mathcal{B}_{a,b,d}$ for any $a\geq b\geq d$ are called \textbf{$2$-standard monomials}. 
\end{definition}

Our goal is now to show that the images of all classical standard monomials are in the span of $2$-standard monomials. We do this by analyzing semistandard tableaux.

\begin{definition}
    Let $a\geq b\geq d$. We say a semistandard tableau $T\in \Tab_{(a+b-d,d)}$ is \emph{$2$-straight} with respect to $(a,b,d)$ if it is in the image of $\Phi_{a,b,d}$. 

    We say $T$ \emph{can be $2$-straightened} with respect to $(a,b,d)$ if there exist some $m\geq 0$ and $2$-straight tableaux $T_1,\ldots, T_m\in \Tab_{(a+b-d,d)}$ such that $\overline{F}_{T,a}$ can be written modulo $\overline{I}^{d+1}$ as
    \[\overline{F}_{T,a}\equiv \sum_{i=1}^m \overline{F}_{T_i,a}\in\overline{I}^d/\overline{I}^{d+1}.\]
    The case where $m=0$ means that $\overline{F}_{T,a}\in\overline{I}^{d+1}$, and this includes if $\overline{F}_{T,a}=0$ in $\overline{R}$.
\end{definition}

In what follows, we characterize $2$-straight tableaux. In \S\ref{sec:span}, we will use this characterization and the $2$-straightening rules to show that all semistandard tableaux can be $2$-straightened.

\subsubsection{$2$-straight tableaux for $a=b=d$}
We characterize semistandard tableaux of shape $(a,a)$ that are $2$-straight with respect to $(a,a,a)$. Let $T\in \Tab_{(a,a)}$ be given by the product of minors $\gamma_1\cdots\gamma_a$, with $\gamma_i=[\alpha_i,\beta_i]$.

\begin{proposition}\label{prop:a=b=d 2-straightened}
    A semistandard tableau $T\in \Tab_{(a,a)}$ is $2$-straight with respect to $(a,a,a)$ if and only if the following conditions hold:
    \begin{enumerate}
    \item An index $i$ does not appear more than two times in $T$, and $\gamma_i\neq\gamma_j$ for all $i\neq j$.
    \item If $\alpha_i=\alpha_{i+1}$ for some $i=2,\ldots,a$, then $\beta_{i-1}<\alpha_i$.
    \item If $\alpha_i=\alpha_{i+1}$ for some $i=1,\ldots,a-1$, then 
    $\beta_j<\alpha_{j+2}$ for
    \[j=\min\{j\geq i:\beta_j\neq\alpha_{j+2}\}\]
    if such a $j$ exists. If such a $j$ does not exist, then $\beta_{a-1}<\beta_a$.
    \item The second row is strictly increasing except possibly in the last column:
    \[\beta_1<\beta_2<\cdots<\beta_{a-2}<\beta_{a-1}\leq\beta_a.\]
    \item If $\beta_{a-1}=\beta_a$, then $\beta_{j-2}<\alpha_j$ for
    \[j=\max\{j\leq a:\beta_{j-2}\neq\alpha_j\}\]
    if such a $j$ exists. If such a $j$ does not exist, then $\alpha_1<\alpha_2$.
\end{enumerate}
\end{proposition}

\begin{proof}
    First suppose $T\in \Tab_{(a,a)}$ is $2$-straight, with $T=\Phi_{a,a,a}(S)$ for $S\in \Tab^{(2)}_{(a,a)}$ with non-identical rows. We show that each of the conditions must hold.
    \begin{enumerate}
        \item Any index $i$ does not appear more than twice in $S$, and therefore does not appear more than twice in $T$. 
        
        Now suppose $\gamma_i=\gamma_j=[\alpha,\beta]$ for some $i\neq j$, with $\alpha<\beta$. Since $T$ is semistandard, we can assume that $j=i+1$. The only way $\alpha$ can appear twice in the first row is if the $i^\text{th}$ column of $S$ has both boxes indexed by $\alpha$. Then since $\gamma_i=[\alpha,\beta]$, this necessarily means that the box of $S$ in the first row after $\alpha$ is $\beta$. Furthermore, since $S$ is $2$-semistandard and $\beta$ needs to appear somewhere in the second row, we must have that the $i^\text{th}$ and $(i+1)^\text{st}$ columns of $S$ are
        \begin{displaymath}
            \ytableausetup
            {mathmode,boxsize=2em}
            \begin{ytableau}
            \alpha & \beta   \\
            \alpha &  \beta 
            \end{ytableau}
            \end{displaymath}
        Then since we have consecutive columns of $S$ with repeated entries in the columns, our construction of $\Phi_{a,a,a}(S)=T$ cannot possibly yield the repeated column $[\alpha,\beta]$ in $T$.
        \item Suppose $\alpha_i=\alpha_{i+1}$ for some $i=2,\ldots,a$. Then the $(i-1),i,(i+1)^{\text{st}}$ columns of $S$ are of the form 
        
            \begin{displaymath}
            \ytableausetup
            {mathmode,boxsize=2em}
            \begin{ytableau}
            \quad & \alpha_i & \beta_i \\
            \beta_{i-1} & \alpha_{i+1} & \quad
            \end{ytableau}
            \end{displaymath}
        since the only way for two indices to appear in the first row is if $\alpha_i=\alpha_{i+1}$ appear in the same column of $S$.
        Therefore, we must have that $\beta_{i-1}<\alpha_{i}=\alpha_{i+1}$.
        \item Suppose that $\alpha_i = \alpha_{i+1}$. This means that $S$ has repeats in some columns. Firstly, suppose that there exists $j=\min\{j\geq i:\beta_j\neq\alpha_{j+2}\}$. Hence, $S$ has a subtableau of the form
        
            \begin{displaymath}
            \ytableausetup
            {mathmode,boxsize=2em}
            \begin{ytableau}
            \alpha_i & \beta_i & \dots & \beta_j & \alpha_{j+2}\\
            \alpha_{i+1} & \alpha_{i+2} & \dots & \beta_{j+1} & \beta_{j+2} 
            \end{ytableau}
            \end{displaymath}
        with $\beta_k = \alpha_{k+2}$ for every $i \leq k < j$. Therefore, we must have that $\beta_{j}<\alpha_{j+2}$. Now suppose that such $j$ does not exist. In this case, $S$ has a subtableau of the form 

            \begin{displaymath}
            \ytableausetup
            {mathmode,boxsize=2em}
            \begin{ytableau}
            \alpha_i & \beta_i & \dots & \beta_{a-2} & \beta_{a-1}\\
            \alpha_{i+1} & \alpha_{i+2} & \dots & \alpha_a & \beta_{a} 
            \end{ytableau}
            \end{displaymath}
with $\beta_k = \alpha_{k+2}$ for every $i \leq k \leq a-2$. This means that $\beta_{a-1}$ can only be equal or smaller than $\beta_a$. But if $\beta_{a-1} = \beta_{a}$, then $T$ is not the image of $S$, as we would have done the other kind of swap to obtain $\Phi_{a,a,a}(S)$. Thus, we have a contradiction and so $\beta_{a-1} < \beta_{a}$. 
            
        \item The second row of $S$ is in strictly increasing order. The only way to introduce a repeat in the second row of $T$ is if $S$ has a repeat in the last column, which forces $\beta_{a-1}=\beta_a$ in $T$.

        \item Suppose $\beta_{a-1} = \beta_a$. This means that $S$ has a repeat in the last column. Firstly, suppose that there exists $j=\max\{j\leq a:\beta_{j-2}\neq\alpha_j\}$. Hence, $S$ has a subtableau of the form
        
            \begin{displaymath}
            \ytableausetup
            {mathmode,boxsize=2em}
            \begin{ytableau}
            \beta_{j-2} & \dots & \beta_{a-2} & \beta_{a-1}\\
            \alpha_{j} & \dots & \alpha_{a}  & \beta_{a}
            \end{ytableau}
            \end{displaymath} 
with $\beta_{k-2} = \alpha_{k}$ for every $j < k \leq a$. Since $\beta_{j-2}\neq\alpha_j$, we must have that $\beta_{j-2}<\alpha_{j}$. Now suppose that such a $j$ does not exist. Then $S$ has the form 

            \begin{displaymath}
            \ytableausetup
            {mathmode,boxsize=2em}
            \begin{ytableau}
            \alpha_1 & \dots & \beta_{a-2} & \beta_{a-1}\\
            \alpha_{2} & \dots & \alpha_{a}  & \beta_{a}
            \end{ytableau}
            \end{displaymath} 
where $\beta_{k-2} = \alpha_{k}$ for every $2 < k \leq a$. This means that $\alpha_1$ can only be equal or smaller than $\alpha_2$. But if $\alpha_1 = \alpha_2$, then this means that the first column of $S$ has $\alpha_1=\alpha_2$ in both boxes. Then, one would need to apply the appropriate swapping procedure from these boxes. Since the rows of $S$ are not identical, there has to be a column somewhere in $S$ without a repeated index, and this would force $\beta_{k-2}<\alpha_k$ for some $k>2$, which contradicts the assumption.
    \end{enumerate}

    Suppose $T$ satisfies the above conditions. We outline how to obtain the $2$-semistandard tableau $S$ such that $\Phi_{a,a,a}(S)=T$. Suppose $\alpha_i=\alpha_{i+1}$ with $i$ minimal. Then swap $\beta_i$ and $\alpha_{i+1}$. Condition (2) ensures that $S$ will be strictly increasing in rows up until column $i$. Continue going from left to right and performing such swaps until the first row is strictly increasing. This process eventually terminates by condition (3). 

    Now, if $\beta_{a-1}=\beta_a$, then swap $\beta_{a-1}$ with $\alpha_a$. If $\beta_{a-2}=\alpha_a$, then swap $\beta_{a-2}$ with $\alpha_{a-1}$. Continue in this way until the bottom row is strictly increasing as well. This process eventually terminates by condition (5). 

    The remaining conditions ensure that $S$ is indeed $2$-semistandard.
\end{proof}

\begin{example}
Let $T$ be given by

        \begin{displaymath}
\ytableausetup
{mathmode,boxsize=1.5em}
\begin{ytableau}
1 & 2 & 2 & 3 \\
2 & 3 & 4 & 5  
\end{ytableau}
\end{displaymath}
Then $T$ is not 2-straight with respect to $(4,4)$ because there is an index appearing more than twice. However, tableaux with this property can be $2$-straightened trivially since $\overline{F}_{T,a}=0$ in $\overline{R}$.
\end{example}

\begin{example}\label{ex:non-straight a=b=d}
Let $T$ be given by

        \begin{displaymath}
\ytableausetup
{mathmode,boxsize=1.5em}
\begin{ytableau}
1 & 2 & 2 & 4 & 5 \\
3 & 3 & 6 & 7 & 7  
\end{ytableau}
\end{displaymath}
Note that $T$ is not 2-straight with respect to $(5,5)$ because it violates conditions (2), (4), and (5) of Proposition~\ref{prop:a=b=d 2-straightened}. We show in Example~\ref{ex:straightening} that $T$ can be $2$-straightened.
\end{example}

\subsubsection{$2$-straight tableaux for other cases}
We now characterize $2$-straight semistandard tableaux for all other cases of $a\geq b\geq d$. The proof follows as in the case $a=b=d$.

\begin{proposition}
    Suppose $a\geq b\geq d$ are such that they are not all equal. A semistandard tableau $T\in \Tab_{(a+b-d,d)}$ is $2$-straight with respect to $(a,b,d)$ if and only if the following conditions hold.
    \begin{enumerate}
    \item An index $i$ does not appear more than two times in $T$, and $\gamma_i\neq\gamma_j$ for $i\neq j$.
    \item If $\alpha_i=\alpha_{i+1}$ for some $i=2,\ldots,d$, then $\beta_{i-1}<\alpha_i$.
    \item If $\alpha_i=\alpha_{i+1}$, then let
    \[j=\min\{i\leq j\leq d:\beta_j\neq\alpha_{j+2}\}\]
    if such a $j$ exists.
    \begin{itemize}
        \item If $a-1=b=d$, then $\beta_j<\alpha_{j+2}$ if $j$ exists.
        \item If $a-1=b-1=d$ and $i=1$, then $\beta_j<\alpha_{j+2}$ if $j$ exists.
        \item In all other cases, then such a $j$ must exist, and $\beta_j<\alpha_{j+2}$.
    \end{itemize}
   
    \item The second row is strictly increasing throughout: $\beta_1<\cdots<\beta_d$.
    \item We have $\alpha_{d+1}<\cdots<\alpha_{a+b-d}$, i.e., the indices corresponding to $1$-minors are strictly increasing.
\end{enumerate}
\end{proposition}

\begin{proof}
    We prove why condition (3) is necessary. The rest of the proof follows as in the $a=b=d$ case. Suppose $T=\Phi_{a,b,d}(S)$ is a $2$-straight semistandard tableau for a $2$-semistandard tableau $S$. Let $j$ be the index as defined in condition (3).

    If $a-1=b=d$, then such a $j$ does not exist if $S$ is of the form 
    \begin{displaymath}
            \ytableausetup
            {mathmode,boxsize=2em}
            \begin{ytableau}
            \alpha'_1 & \dots & \alpha'_{a-2} & \alpha'_{a-1} & \alpha'_{a}\\
            \alpha'_1 & \dots & \alpha'_{a-2}  & \alpha'_{a-1}
            \end{ytableau}
            \end{displaymath} 

    If $a-1=b-1=d$ and $i=1$, then again $j$ does not exist if $S$ is of shape $(a,a)$ and has identical rows.

    Now suppose we are in any other case, i.e., if $a-1=b-1=d$ with $i>1$, or other possible $a,b,d$. If such a $j$ does not exist, then this would imply that $S$ has columns $i$ through $d+2$ of the form
    \begin{displaymath}
            \ytableausetup
            {mathmode,boxsize=2em}
            \begin{ytableau}
            \alpha_i & \alpha_{i+2} & \alpha_{i+3} & \dots & \alpha_{d} & \alpha_{d+1} & \alpha_{d+2} & \alpha_{d+2}\\
            \alpha_i & \alpha_{i+2} & \alpha_{i+3} & \dots & \alpha_{d}  & \alpha_{d+1} 
            \end{ytableau}
            \end{displaymath} 
    But then $S$ would not be $2$-semistandard, as the first row would not be strictly increasing.
\end{proof}

\section{Span of 2-standard monomials}\label{sec:span}

Recall that $\mathcal{B}_{a,b,d}$ is the set of $2$-standard monomials in $\overline{I}^d_{(a,b)}$. The main result of this section is the following.

\begin{theorem}\label{thm:main span}
    For $a\geq b\geq d$, the residue classes of the elements of $\mathcal{B}_{a,b,d}$ span the vector space $(\overline{I}^d/\overline{I}^{d+1})_{(a,b)}$.
\end{theorem}

We consider the following cases:
\begin{enumerate}
    \item $a=b=d$;
    \item $a-1=b=d$;
    \item $a-1=b-1=d$;
    \item all other possible $a\geq b\geq d$.
\end{enumerate}
We consider case (1) in \S \ref{subsec:span a=b=d}, case (2) in \S \ref{subsec:span a-1=b=d}, and cases (3) and (4) in \S \ref{subsec:span other cases}. For each case, we show that every semistandard tableau $T\in\Tab_{(a+b-d,d)}$ can be $2$-straightened with respect to $(a,b,d)$ by induction on the number of columns. This proves the theorem, since the residue classes modulo $\overline{I}^{d+1}$ of $\overline{F}_{T,a}$ for $T\in\Tab_{(a+b-d,d)}$ span $(\overline{I}^d/\overline{I}^{d+1})_{(a,b)}$.

If $T\in\Tab_{(a+b-d,d)}$ and $1\leq i\leq a+b-d$, then let $T^{(i)}$ denote the subtableau of $T$ consisting of the first $i$ columns of $T$. We also conflate a semistandard tableau $T$ with the image of its standard monomial $\overline{F}_{T,a}$ in $\overline{R}$.

\subsection{Case (1): $a=b=d$}\label{subsec:span a=b=d}

We proceed by induction on $a$. First, note that for any $T\in\Tab_{(a,a)}$ for $a=1,2$, either $\overline{F}_{T,a}=0$ in $\overline{R}$, or $T$ is already $2$-straight.

We now assume that $a\geq 3$, and that the desired statement holds for $a-1$. For a tableau $T\in\Tab_{(a,a)}$, we want to assume that $T^{(a-1)}$ is $2$-straight. However, the $2$-straightening process for $T^{(a-1)}$ could cause $T$ to no longer be semistandard. The following key lemma allows us to assume that this is not the case.

\begin{lemma}\label{lem:a=b=d trunc 2-ss}
    Let $a\geq 3$, and $T\in \Tab_{(a,a)}$. Further suppose that semistandard tableaux of shape $(a-1,a-1)$ can be $2$-straightened. Then $\overline{F}_{T,a}$ can be written modulo $\overline{I}^{d+1}$ as
    \[\overline{F}_{T,a}\equiv \sum_{i=1}^m\overline{F}_{T_i,a},\]
    where $T_i\in \Tab_{(a,a)}$ is semistandard and $T_i^{(a-1)}$ is $2$-straight.
\end{lemma}

\begin{proof}
    Let $T=\gamma_1\cdots\gamma_a$. By assumption, $T^{(a-1)}=\gamma_1\cdots\gamma_{a-1}$ can be $2$-straightened and therefore can be expressed as a sum of $2$-straight tableaux in $\Tab_{(a-1,a-1)}$. Substituting this expression for $\gamma_1\cdots\gamma_{a-1}$ yields a representation of $T$ as a linear combination of tableaux $S_1,\ldots,S_r$ of shape $(a,a)$. 

    If such a tableau $S_i$ is not semistandard, then apply classical straightening to express it as a sum of semistandard tableaux $S_{i,1},\ldots,S_{i,r'}$. For each $S_{i,j}$, again $2$-straighten the first $(a-1)$ columns.

    We show that an iteration of this process must eventually yield a linear combination of semistandard tableaux of the desired form. Let $\gamma_a=[\alpha_a,\beta_a]$ be the last column of $T$. Any time classical straightening is applied, $\gamma_a$ is replaced with $[\alpha'_a,\beta_a]$, where $\alpha'_a>\alpha_a$. Since the indices of $T$ are fixed, this process must eventually terminate. In the worst-case scenario, $\alpha'_a,\beta_a$ are the two maximal indices appearing in $T$, and so after $2$-straightening the first $(a-1)$ columns, we obtain tableaux of the desired form.
\end{proof}

It therefore suffices to show that if $T\in \Tab_{(a,a)}$ with $T^{(a-1)}$ already $2$-straight, then $T$ can be $2$-straightened. 
For the rest of this subsection, suppose $T=\gamma_1\cdots\gamma_a$ with $\gamma_i=[\alpha_i,\beta_i]$. We further suppose that no index $i$ appears more than twice in $T$, and that $\gamma_i\neq\gamma_j$ for $i\neq j$; otherwise, $\overline{F}_{T,a}=0$ in $\overline{R}$. 
If $T$ is not already $2$-straight, then we have the following cases to consider:
\begin{enumerate}
    \item $\alpha_{a-1}=\alpha_a$;
    \item $\beta_{a-2}=\beta_{a-1}$;
        \item $\alpha_{a-2}=\alpha_{a-1}$;
    \item $\alpha_i=\alpha_{i+1}$ for some $i\leq a-2$, and $\beta_{j}=\alpha_{j+2}$ for $j=i,\ldots,a-3$;
        \item $\beta_{a-1}=\beta_a$.
\end{enumerate}

\begin{lemma}\label{lem:a=b=d alpha_d-1=alpha_d}
    If $\alpha_{a-1}=\alpha_a$, then $T$ can be $2$-straightened.
\end{lemma}

\begin{proof}
    If $\alpha_{a-2}=\alpha_{a-1}$ or $\beta_{a-1}=\beta_a$, then $T= 0$, so suppose $\alpha_{a-2}<\alpha_{a-1}$ and $\beta_{a-1}<\beta_a$. Furthermore, if $\beta_{a-2}<\alpha_a$, then $T$ is $2$-straight, and if $\beta_{a-2}=\alpha_a$, then $T=0$. Thus, we assume that $\beta_{a-2}>\alpha_a$. The following argument holds with $\beta_{a-2}\leq\beta_{a-1}$, so we have
    \[\alpha_{a-2}<\alpha_{a-1}=\alpha_a<\beta_{a-2}\leq\beta_{a-1}<\beta_a.\]

    By Corollary~\ref{cor:2-straightening swap 3,3} applied to $\gamma_{a-2}\gamma_{a-1}\gamma_a$, we have that $T$ is equal to
        \begin{displaymath}
        \ytableausetup
        {mathmode,boxsize=2em}
        \gamma_1\cdots\gamma_{a-3}\,\,\begin{ytableau}
        \alpha_{a-2} & \alpha_{a} & \beta_{a-1} \\
        \alpha_{a-1} & \beta_{a-2} & \beta_a
        \end{ytableau}
        \end{displaymath}
    This tableau may not be semistandard; if not, then apply classical straightening and then $2$-straightening of the first $(a-1)$ columns to write $T$ modulo $\overline{I}^{d+1}$ as a sum of semistandard tableaux of the form
    \begin{displaymath}
        \ytableausetup
        {mathmode,boxsize=2em}
        S = \gamma'_1\cdots\gamma'_{a-3}\,\,\begin{ytableau}
        \alpha'_{a-2} & \alpha'_{a-1} & \beta_{a-1} \\
        \beta'_{a-2} & \beta_{a-2} & \beta_a
        \end{ytableau}
        \end{displaymath}
    Note that $\beta_{a-2},\beta_{a-1},\beta_a$ must be in these positions of $S$ since they are the largest indices appearing in $T$. We show that $S$ is $2$-straight.

    Condition (1) holds by assumption. We have that $\beta'_{a-2},\alpha'_{a-1}<\beta_{a-1}$, so conditions (2) and (3) hold. We have that $\beta_{a-2}<\beta_a$, so condition (5) does not apply. Thus, it suffices to show that $\beta'_{a-2}<\beta_{a-2}$ to show that condition (4) holds. If $\beta_{a-2}=\beta_{a-1}$, then $\beta'_{a-2}<\beta_{a-2}$ since the same index cannot appear more than twice in $T$. If $\beta_{a-2}<\beta_{a-1}$, then the only way for $\beta'_{a-2}=\beta_{a-2}$ is if $\beta_{a-3}=\beta_{a-2}$ in $T$, as $\beta_{a-2}$ is greater than all $\alpha$ indices. However, since $T^{(a-1)}$ was assumed to be $2$-straight, it must have been that $\beta_{a-3}<\beta_{a-2}$. Thus, condition (4) holds for $S$, and so $S$ is $2$-straight. 
\end{proof}

\begin{lemma}\label{lem:a=b=d beta_a-2=beta_a-1}
    If $\beta_{a-2}=\beta_{a-1}$, then $T$ can be $2$-straightened.
\end{lemma}

\begin{proof}
    If $\alpha_{a-2}=\alpha_{a-1}$ or $\beta_{a-1}=\beta_a$, then $T= 0$, so assume $\alpha_{a-2}<\alpha_{a-1}$ and $\beta_{a-1}<\beta_a$. If $\alpha_{a-1}=\alpha_a$, then we considered this case in Lemma~\ref{lem:a=b=d alpha_d-1=alpha_d}, so also assume $\alpha_{a-1}<\alpha_a$. Finally, if $\beta_{a-1}=\alpha_a$, then $T= 0$. We therefore either have that $\beta_{a-2}>\alpha_a$ or $\beta_{a-2}<\alpha_a$.

    If $\beta_{a-2}>\alpha_a$, we have that
    \[\alpha_{a-2}<\alpha_{a-1}<\alpha_a<\beta_{a-2}=\beta_{a-1}<\beta_a.\]
    By Corollary~\ref{cor:2-straightening swap 3,3} applied to $\gamma_{a-2}\gamma_{a-1}\gamma_a$, we have that $T$ is equivalent to
    \begin{displaymath}
        \ytableausetup
        {mathmode,boxsize=2em}
        \gamma_1\cdots\gamma_{a-3}\,\,\begin{ytableau}
        \alpha_{a-2} & \alpha_{a} & \beta_{a-1} \\
        \alpha_{a-1} & \beta_{a-2} & \beta_a
        \end{ytableau}
        \end{displaymath}
    This tableau may not be semistandard; if not, then apply classical straightening and then $2$-straightening of the first $(a-1)$ columns to write $T$ modulo $\overline{I}^{d+1}$ as a sum of semistandard tableaux of the form
    \begin{displaymath}
        \ytableausetup
        {mathmode,boxsize=2em}
        S = \gamma'_1\cdots\gamma'_{a-3}\,\,\begin{ytableau}
        \alpha'_{a-2} & \alpha'_{a-1} & \beta_{a-1} \\
        \beta'_{a-2} & \beta_{a-2} & \beta_a
        \end{ytableau}
        \end{displaymath}
    Note that $\beta_{a-2},\beta_{a-1},\beta_a$ must be in these positions of $S$ since they are the largest indices appearing in $T$. We have that $S$ is $2$-straight; the argument follows as in Lemma~\ref{lem:a=b=d alpha_d-1=alpha_d}.

    If $\beta_{a-2}<\alpha_a$, we have that
    \[\alpha_{a-2}<\alpha_{a-1}<\beta_{a-2}=\beta_{a-1}<\alpha_a<\beta_a.\]
    Note in this case that $\beta_{a-2}=\beta_{a-1},\alpha_a,\beta_a$ are strictly greater than any other indices appearing in $T$.
    By Corollary~\ref{cor:2-straightening swap 3,3} applied to $\gamma_{a-2}\gamma_{a-1}\gamma_a$, we have that $T$ is equivalent to
    \begin{displaymath}
        \ytableausetup
        {mathmode,boxsize=2em}
        \gamma_1\cdots\gamma_{a-3}\,\,\begin{ytableau}
        \alpha_{a-2} & \beta_{a-2} & \beta_{a-1} \\
        \alpha_{a-1} & \alpha_a & \beta_a
        \end{ytableau}
        \end{displaymath}
    This tableau may not be semistandard; if not, then apply classical straightening and then $2$-straightening of the first $(a-1)$ columns to write $T$ as a sum of semistandard tableaux of the form
    \begin{displaymath}
        \ytableausetup
        {mathmode,boxsize=2em}
        S = \gamma'_1\cdots\gamma'_{a-3}\,\,\begin{ytableau}
        \alpha'_{a-2} & \beta_{a-2} & \beta_{a-1} \\
        \beta'_{a-2} & \alpha_a & \beta_a
        \end{ytableau}
        \end{displaymath}
    We show that $S$ is $2$-straight. Condition (1) holds by assumption. We have that $\beta'_{a-2}<\alpha_a<\beta_a$ so conditions (4) and (5) hold. Finally, we have that $\beta'_{a-2}<\beta_{a-1}$, so conditions (2) and (3) hold.
\end{proof}

\begin{lemma}\label{lem:a=b=d alpha_d-2=alpha_d-1}
    If $\alpha_i=\alpha_{i+1}$ for some $1\leq i\leq a-2$ and $\beta_j=\alpha_{j+2}$ for $j=i,\ldots,a-3$, then $T$ can be $2$-straightened. In particular, if $\alpha_{i-2}=\alpha_{i-1}$, then $T$ can be $2$-straightened.
\end{lemma}

\begin{proof}
    If $\alpha_{a-1}=\alpha_a$ or $\beta_{a-2}=\beta_{a-1}$, then $T=0$, so suppose $\alpha_{a-1}<\alpha_a$ and $\beta_{a-2}<\beta_{a-1}$.

    First suppose $\beta_{a-1}<\beta_a$. If $\beta_{a-2}\leq\alpha_a$, then $T$ is $2$-straight, so suppose $\beta_{a-2}>\alpha_a$. Then, we have
    \[\alpha_i=\alpha_{i+1}<\beta_i=\alpha_{i+2}<\cdots < \beta_{a-3}=\alpha_{a-1}<\alpha_a<\beta_{a-2}<\beta_{a-1}<\beta_a.\]
    By Corollary~\ref{cor:2-straightening pairs}, $\gamma_i\cdots\gamma_a$ is equivalent to
    \begin{displaymath}
        \ytableausetup
        {mathmode,boxsize=2em}
        \gamma_i\cdots\gamma_{a-3}\left(\begin{ytableau}
        \alpha_{a-2} & \alpha_{a-1} & \beta_{a-2} \\
        \alpha_a & \beta_{a-1} & \beta_a
        \end{ytableau}
        +
        \begin{ytableau}
        \alpha_{a-2} & \alpha_{a-1} & \beta_{a-1} \\
        \alpha_a & \beta_{a-2} & \beta_a
        \end{ytableau}\right)
        \end{displaymath}
    We show that substituting this expression for $\gamma_i\cdots\gamma_a$ in $T$ results in $2$-straight tableaux. We consider the first tableau
    \begin{displaymath}
        \ytableausetup
        {mathmode,boxsize=2em}
        \gamma_1\cdots\gamma_{a-3}\,\,\begin{ytableau}
        \alpha_{a-2} & \alpha_{a-1} & \beta_{a-2} \\
        \alpha_a & \beta_{a-1} & \beta_a
        \end{ytableau}
        \end{displaymath}
    Note that this is semistandard: if $i<a-2$, then $\beta_{a-3}=\alpha_{a-1}<\alpha_a$ by our assumption; if $i=a-2$, then $\beta_{a-3}<\alpha_{a-1}$ by the inductive hypothesis that the first $(a-1)$ columns form a $2$-straight tableau. Next, we still have $\alpha_i=\alpha_{i+1}$ and $\beta_{i-1}<\alpha_i$, as these indices were not moved. In addition, we still have that $\beta_j=\alpha_{j+2}$ for $j=i,\ldots,a-3$, and now we have $\alpha_a<\beta_{a-2}$, so condition (3) holds. Finally, we have that the second row is strictly increasing throughout, so conditions (4) and (5) hold.

    A similar argument shows that the second tableau in the expression is also $2$-straight.

    Now suppose $\beta_{a-1}=\beta_a$. If $\beta_{a-2}<\alpha_a$, then $T$ is $2$-straight, so first suppose $\beta_{a-2}=\alpha_a$. If $i=1$, then $T= 0$. If $i>1$, then we show that $T$ is already $2$-straight. We have that $\beta_{j-2}=\alpha_j$ for $j=a,a-1,\ldots,i+3$. In addition, since $T^{(a-1)}$ is $2$-straight, we have that $\beta_{i-1}<\alpha_{i+1}$. Thus, condition (5) holds, while the other conditions hold by induction.

    Now suppose $\beta_{a-1}=\beta_a$ and $\beta_{a-2}>\alpha_a$. Then $\gamma_i\cdots\gamma_a$ is equivalent to
    \begin{displaymath}
        \ytableausetup
        {mathmode,boxsize=2em}
        \gamma_i\cdots\gamma_{a-3}\,\,\begin{ytableau}
        \alpha_{a-2} & \alpha_{a-1} & \beta_{a-2} \\
        \alpha_a & \beta_{a-1} & \beta_a
        \end{ytableau}
        \end{displaymath}
    Substituting this expression for $\gamma_i\cdots\gamma_a$ in $T$ results in a $2$-straight tableau: the argument for showing conditions (1)--(4) hold follow as in the $\beta_{a-1}<\beta_a$ case, and for condition (5), we have that $\alpha_a<\beta_{a-2}$.
\end{proof}

\begin{lemma}\label{lem:a=b=d beta_a-1=beta_a}
    If $\beta_{a-1}=\beta_a$, then $T$ can be $2$-straightened.
\end{lemma}

\begin{proof}
    Let $j=\max\{j\leq a:\beta_{j-2}\neq\alpha_j\}$. If such a $j$ does not exist, then we either have that $\alpha_1=\alpha_2$, in which case $T=0$, or $\alpha_1<\alpha_2$, in which case $T$ is $2$-straight. Thus, assume that $j$ does exist. We further assume that $\alpha_{j-2}<\alpha_{j-1}$; if $\alpha_{j-2}=\alpha_{j-1}$, then $T$ is $2$-straight by induction on $T^{(a-1)}$.

    If $\beta_{j-2}<\alpha_j$, then $T$ is $2$-straight, so assume $\beta_{j-2}>\alpha_j$. In this case, we have that
    \[\beta_{j-2}<\beta_{j-1}=\alpha_{j+1}<\cdots<\beta_{a-2}=\alpha_a<\beta_{a-1}<\beta_a\]
    are strictly greater than any other indices appearing in $T$.
    By Corollary~\ref{cor:2-straightening pairs}, $\gamma_{j-2}\cdots\gamma_a$ can be expressed as
    \begin{displaymath}
        \ytableausetup
        {mathmode,boxsize=2em}
        \left(\begin{ytableau}
        \alpha_{j-2} & \alpha_j & \beta_{j-2} \\
        \alpha_{j-1} & \beta_{j-1} & \beta_j
        \end{ytableau}
        +
        \begin{ytableau}
        \alpha_{j-2} & \alpha_{j-1} & \beta_{j-2} \\
        \alpha_j & \beta_{j-1} & \beta_j
        \end{ytableau}\right)\gamma_{j+1}\cdots\gamma_a.
        \end{displaymath}
    Substitute this expression for $\gamma_{j-2}\cdots\gamma_a$ in $T$. The resulting tableaux may not be semistandard; if not, then apply classical straightening and then $2$-straightening, both to the first $(j-1)$ columns. We show that the resulting tableaux are $2$-semistandard.

    A tableau appearing in the expression is of the form
    \begin{displaymath}
        \ytableausetup
        {mathmode,boxsize=2em}
        S = \gamma'_1\cdots\gamma'_{j-3}\,\,\begin{ytableau}
        \alpha'_{j-2} & \alpha'_{j-1} & \beta_{j-2} \\
        \beta'_{j-2} & \beta_{j-1} & \beta_j
        \end{ytableau}\,\,\gamma_{j+1}\cdots\gamma_a.
        \end{displaymath}
    Note that $\beta_{j-1}$ must appear in that position, since it is strictly greater than all other indices appearing in the first $(j-1)$ columns. Furthermore, we have that $\alpha'_{j-1}<\beta_{j-2}$ and $\beta_{j-1}\leq\beta_j$, with equality occurring if and only if $j=a$. This shows that $S$ is semistandard. We also have that $S$ is $2$-straight, as we must have that $\alpha'_{j-1},\beta'_{j-2}<\beta_{j-2},\beta_{j-1}$.
\end{proof}

The results of this subsection prove Theorem~\ref{thm:main span} in the case that $a=b=d$.

\begin{example}\label{ex:straightening}
In this example, we show the 2-straightening procedure for $T$ given in Example~\ref{ex:non-straight a=b=d} by

\begin{displaymath}
\ytableausetup
{mathmode,boxsize=1.5em}
\begin{ytableau}
1 & 2 & 2 & 4 & 5 \\
3 & 3 & 6 & 7 & 7  
\end{ytableau}
\end{displaymath}
Since $T^{(2)}$ is 2-straight, we first use Lemma~\ref{lem:a=b=d alpha_d-1=alpha_d} to 2-straighten $T^{(3)}$ and obtain $T_1$ given by

\begin{displaymath}
\ytableausetup
{mathmode,boxsize=1.5em}
\begin{ytableau}
1 & 2 & 3 & 4 & 5 \\
2 & 3 & 6 & 7 & 7  
\end{ytableau}
\end{displaymath}
Notice that $T_1$ is still semistandard and satisfies all conditions of Proposition~\ref{prop:a=b=d 2-straightened}, except for (5). Since $T_1^{(4)}$ is 2-straight, we may apply Lemma~\ref{lem:a=b=d beta_a-1=beta_a} to 2-straighten $T_1$; let $T_2,T_3$ be given by

\begin{displaymath}
\ytableausetup
{mathmode,boxsize=1.5em}
T_2 =\begin{ytableau}
1 & 2 & 3 & 5 & 6 \\
2 & 3 & 4 & 7 & 7  
\end{ytableau}
\end{displaymath}

\begin{displaymath}
\ytableausetup
{mathmode,boxsize=1.5em}
T_3=\begin{ytableau}
1 & 2 & 3 & 4 & 6 \\
2 & 3 & 5 & 7 & 7  
\end{ytableau}
\end{displaymath}
Observe that $T_2$ and $T_3$ are semistandard and 2-straight. One can check that $\overline{F}_{T,5}=\overline{F}_{T_1,5}=\overline{F}_{T_2,5}+\overline{F}_{T_3,5}$.
\end{example}

\subsection{Induction step for $a-1=b=d$ case}\label{subsec:span a-1=b=d}
We now consider the case where $a-1=b=d$. The arguments are similar to those of the previous case with $a=b=d$.

\begin{lemma}
    Let $a\geq 2$, and $T\in \Tab_{(a,a-1)}$. Then $\overline{F}_{T,a}$ can be written modulo $\overline{I}^{d+1}$ as
    \[\overline{F}_{T,a}\equiv\sum_{i=1}^m\overline{F}_{T_i,a},\]
    where $T_i\in \Tab_{(a,a)}$ and $T_i^{(d)}\in \Tab_{(a,a)}$ is $2$-straight.
\end{lemma}

\begin{proof}
    This follows as for the proof of Lemma~\ref{lem:a=b=d trunc 2-ss}. Each time classical straightening must be applied, the index in the $a^\text{th}$ column strictly increases, and so the process of straightening and $2$-straightening must eventually terminate.
\end{proof}

It therefore suffices to show that if $T\in \Tab_{(a,a-1)}$ with $T^{(a-1)}$ already $2$-straight, then $T$ can be $2$-straightened. We further suppose that no index $i$ appears more than twice in $T$. If $T$ is not already $2$-straight, then we have the following cases to consider:
\begin{enumerate}
    \item $\alpha_{a-1}=\alpha_{a}$;
    \item $\beta_{a-2}=\beta_{a-1}$;
    \item $\alpha_i=\alpha_{i+1}$ for some $i\leq a-2$, and $\beta_j=\alpha_{j+2}$ for $j=i,\ldots,a-3$.
\end{enumerate}

In each of these cases, one can adapt the proofs from the $a=b=d$ case (Lemmas~\ref{lem:a=b=d alpha_d-1=alpha_d}, \ref{lem:a=b=d beta_a-2=beta_a-1}, \ref{lem:a=b=d alpha_d-2=alpha_d-1}). The differences are that the box corresponding to $\beta_a$ is removed, and one needs to show that the resulting tableaux all have strictly increasing bottom rows, which holds in all of the cases.

This proves Theorem~\ref{thm:main span} in the case of $a-1=b=d$.

\subsection{Induction step for other cases}\label{subsec:span other cases}
We now consider all other cases of $a\geq b\geq d$. In particular, if $a>b$, then $a-d\geq 2$. The arguments from the previous subsection can be applied here as well to show that if conditions (1), (2), or (4) do not hold, then a semistandard tableau can be $2$-straightened.

The following lemmas (Lemma~\ref{lem:repeat indices in tails} and \ref{lem:repeat indices tails special}) show that if condition (5) does not hold for a semistandard tableau $T$, then $\overline{F}_{T,a}\equiv 0$ in $\overline{I}^d/\overline{I}^{d+1}$.

\begin{lemma}\label{lem:repeat indices in tails}
    Suppose $a>b\geq d$ with $a-d\geq 2$, or $a-2=b-2\geq d$. If $T\in \Tab_{(a+b-d,d)}$ with $\alpha_i=\alpha_{i+1}$ for some $i\geq d+1$, then $\overline{F}_{T,a}\equiv 0$ in $\overline{I}^d/\overline{I}^{d+1}$.
\end{lemma}

\begin{proof}
    If $i\in\{d+1,\ldots,a-1\}$, then $x_{\alpha_i}^2$ divides $F_{T,a}$, so its image is zero in $\overline{R}$. Similarly, if $i\in\{a+1,\ldots,a+b-d-1\}$, then $y_{\alpha_i}^2$ divides $F_{T,a}$.

    Now suppose $\alpha_a=\alpha_{a+1}$. Then the image of 
    \[x_{\alpha_{d+1}}x_{\alpha_{d+2}}\cdots x_{\alpha_a}y_{\alpha_a}y_{\alpha_{a+2}}\cdots y_{\alpha_{a+b-d}}\]
    is equal to the image of
    \[(x_{\alpha_{d+1}}y_{\alpha_a}+x_{\alpha_a}y_{\alpha_{d+1}})x_{\alpha_{d+2}}\cdots x_{\alpha_a}y_{\alpha_{a+2}}\cdots y_{\alpha_{a+b-d}}.\]
    This latter element is in $I$, since $d+1<a$ and so $[\alpha_{d+1},\alpha_a]$ is indeed a $2$-minor. Thus, $\overline{F}_{T,a}=\overline{F}_{T',a+1}$, where $T'$ is obtained from $T$ by moving a box containing $\alpha_a$ to the end of the second row.
\end{proof}

\begin{lemma}\label{lem:repeat indices tails special}
    Suppose $a-1=b-1=d$. If $T\in \Tab_{(a+1,a-1)}$ with $\alpha_a=\alpha_{a+1}$, then $\overline{F}_{T,a}\in\overline{I}^{d+1}$.
\end{lemma}

\begin{proof}
    We have that
    \[[\alpha_{a-1},\beta_{a-1}]x_{\alpha_a}y_{\alpha_a}=[\alpha_{a-1},\alpha_a][\beta_{a-1},\alpha_a].\]
    Thus, if $S$ is the tableau of shape $(a,a)$ given by
    \begin{displaymath}
        \ytableausetup
        {mathmode,boxsize=2em}
        S = \gamma_1\cdots\gamma_{a-2}\,\,\begin{ytableau}
        \alpha_{a-1} & \beta_{a-1} \\
        \alpha_a & \alpha_a
        \end{ytableau}
        \end{displaymath}
    then we have 
    \[\overline{F}_{T,a}=\overline{F}_{S,a}\in\overline{I}^{d+1}.\]
\end{proof}

The following lemmas show that if condition (3) does not hold (in the cases besides $a-1=b=d$), then the tableau can be $2$-straightened.

\begin{lemma}
    Let $m\geq 1$ and $\gamma\neq\delta\in [n]$. Suppose $T\in \Tab_{(m+3,m+1)}$ is given by
    \begin{displaymath}
    \ytableausetup
    {mathmode,boxsize=2em}
    \begin{ytableau}
    \gamma & \alpha & \alpha & \beta_1 & \beta_2 & \cdots & \cdots & \beta_{m-1} &  \beta_m  \\
    \delta & \beta_1 & \beta_2 & \cdots & \cdots & \cdots & \beta_m 
    \end{ytableau}
    \end{displaymath}
    Then
    \[\overline{F}_{T,m+2}\in\overline{I}^{m+2}.\]
    In particular, if $a-1=b-1=d$ and condition (3) does not hold because $\alpha_i\neq 1$ and such an index $j$ does not exist, then $\overline{F}_{T,a}\equiv0$ in $\overline{I}^d/\overline{I}^{d+1}$.
\end{lemma}

\begin{proof}
    Let $S\in \Tab_{(m+2,m+2)}$ be given by
    \begin{displaymath}
    \ytableausetup
    {mathmode,boxsize=2em}
    \begin{ytableau}
    \alpha & \alpha & \beta_1 & \beta_2 & \cdots & \cdots & \beta_{m-1} &  \beta_m  \\
    \beta_1 & \beta_2 & \cdots & \cdots & \cdots & \beta_m & \gamma & \delta
    \end{ytableau}
    \end{displaymath}
    Then, by Proposition~\ref{prop:main 2-straightening rule},
    \[\overline{F}_{T,m+2}=\overline{F}_{S,m+1}=[\gamma,\delta]x_\alpha y_\alpha\prod_{i=1}^m x_{\beta_i}y_{\beta_i}.\]
\end{proof}

\begin{lemma}\label{lem:x gamma index}
    Let $m\geq 1$, and let $T\in\Tab_{m+3,m+1}$ be given by
    \begin{displaymath}
    \ytableausetup
    {mathmode,boxsize=2em}
    \begin{ytableau}
    \alpha & \alpha & \beta_1 & \beta_2 & \cdots & \cdots & \beta_{m-1} &  \beta_m & \delta \\
    \beta_1 & \beta_2 & \cdots & \cdots & \cdots & \beta_m &\gamma
    \end{ytableau}
    \end{displaymath}
    Then 
    \begin{align*}
        \overline{F}_{T,m+3}&=x_\gamma x_\delta\left(x_\alpha y_\alpha\prod_{i=1}^m x_{\beta_i}y_{\beta_i}\right),\\
        \overline{F}_{T,m+2}&=x_\gamma y_\delta\left(x_\alpha y_\alpha\prod_{i=1}^m x_{\beta_i}y_{\beta_i}\right).
    \end{align*}
\end{lemma}

\begin{proof}
    This follows from Proposition~\ref{prop:main 2-straightening rule}, and noting that
    \[[\beta_m,\gamma]x_{\beta_m}x_\delta=x_{\beta_m}x_\gamma x_\delta y_{\beta_m},\quad [\beta_m,\gamma]x_{\beta_m} y_\delta=x_{\beta_m}x_\gamma y_{\beta_m}y_\delta.\]
\end{proof}

\begin{lemma}
    Suppose $a-d\geq 2$ and $d\geq 1$. Let $T\in \Tab_{(a+b-d,d)}$ with
    \begin{itemize}
        \item $\alpha_1=\alpha_2$;
        \item $\beta_i=\alpha_{i+2}$ for $i=1,\ldots,d-1$;
        \item rows are strictly increasing.
    \end{itemize}
    Then $T$ can be $2$-straightened.
\end{lemma}

\begin{proof}
    If $\beta_d<\alpha_{d+2}$, then $T$ is already $2$-straight. If $\beta_d=\alpha_{d+2}$, then $\overline{F}_{T,a}=0$.

    Now suppose $\beta_d>\alpha_{d+2}$. We have the following four cases:
    \begin{enumerate}
        \item $\beta_d<\alpha_{a+1}$, with $\beta_d\neq\alpha_i$ for all $i=d+3,\ldots,a$;
        \item $\beta_d> \alpha_{a+1}$, with $\beta_d\neq \alpha_i$ for all $i=d+3,\ldots,a+1$;
        \item $\beta_d\leq \alpha_a$, with $\beta_d=\alpha_i$ for some $i=d+3,\ldots,a$;
        \item $\beta_d\geq\alpha_{a+1}$, with $\beta_d=\alpha_i$ for some $i=a+1,\ldots,a+b-d$.
    \end{enumerate}

    \noindent\textbf{Case (1)}: Let $T'$ be the tableau obtained from $T$ by removing the box with $\alpha_{d+2}$ from the first row, replacing $\beta_d$ with $\alpha_{d+2}$ in the second row, and putting a box with $\beta_d$ in the first row so that the first row is increasing. Then $T'$ is $2$-straight of the same shape, as we have that $\alpha_{d+2}<\alpha_{d+3}$, and we have
    \[\overline{F}_{T,a}=\overline{F}_{T',a}.\]
    \noindent\textbf{Case (2)}: Let $S$ be the tableau of shape $(a+b-d-1,d+1)$ obtained from $T$ by moving the box with index $\alpha_{a+1}$ to the end of the second row. Let $T'\in\Tab_{(a+b-d,d)}$ be the tableau obtained from $T$ by removing the box with $\alpha_{d+2}$ from the first row, replacing $\beta_d$ with $\alpha_{d+2}$ in the second row, and putting a box with $\beta_d$ in first row so that the first row remains increasing; in particular, the last $b-d$ entries of the first row are $\alpha_{a+2},\ldots,\alpha_{a+b-d},\beta_d$. Then, $T'$ is $2$-straight of the same shape as $T$, and we have that
    \[\overline{F}_{T,a}+\overline{F}_{T',a}=\overline{F}_{S,a}.\]
    Since $\overline{F}_{S,a}\in\overline{I}^{d+1}$ and $\overline{F}_{T',a}\in\overline{I}^d$, this shows that $\overline{F}_{T,a}\equiv\overline{F}_{T',a}$ in $\overline{I}^d/\overline{I}^{d+1}$. Thus, $T$ can be $2$-straightened.

    \noindent\textbf{Case (3)}: By Lemma~\ref{lem:x gamma index} with $\gamma=\beta_d$, we have that $\overline{F}_{T,a}$ is divisible by $x_{\beta_d}^2$, and so it is equal to $0$ in $\overline{R}$.

    \noindent\textbf{Case (4)}: Let $S$ be the tableau of shape $(a+b-d-2,d)$ obtained from $T$ by removing the boxes with indices $\alpha_a$ and $\alpha_i$ from the first row (recall that $\alpha_i=\beta_d$, and $i\geq a+1$). Then
    \begin{align*}
        \overline{F}_{T,a}&=\overline{F}_{S,a-1}[\alpha_a,\alpha_i],
        \end{align*}
        as both are equal to
        \[(x_{\beta_d}x_{\alpha_{d+2}}x_{\alpha_{d+3}}\cdots x_{\alpha_{a-1}})(x_{\alpha_a}y_{\alpha_i})(y_{\alpha_{a+1}}y_{\alpha_{a+2}}\cdots\widehat{y_{\alpha_i}}\cdots y_{\alpha_{a+b-d}})\prod_{i=2}^{d+1}x_{\alpha_i}y_{\alpha_i}\]
        by Lemma~\ref{lem:x gamma index}.

    Since $\overline{F}_{S,a-1}[\alpha_a,\alpha_i]\in\overline{I}^{d+1}$, this shows that $\overline{F}_{T,a}\in\overline{I}^{d+1}$ as well.
\end{proof}

Combining the above results, we have proved Theorem~\ref{thm:main span} for all other possible $a\geq b\geq d$, including the case that $a-1=b-1=d$.

\section{Proof of main result and future directions}\label{sec:main thm}

In this section, we prove Theorem~\ref{thm:main} on the character formulas for $(\overline{I}^d/\overline{I}^{d+1})_{(a,b)}$ for any $a\geq b\geq d$. We first provide some definitions and results that will apply to the $a=b$ case.

\begin{definition}
    Let $a\geq 1$ and $0\leq i\leq a$. Let $T\in \Tab^{(2)}_{(a+i,a-i)}$ be a $2$-semistandard tableau. Define the set of $2$-semistandard tableau $\mathcal{D}_{a,i}$ as follows: $T\in\Tab^{(2)}_{(a+i,a-i)}$ is in $\mathcal{D}_{a,i}$ if and only if there exists an index $j$, for some $1\leq j\leq a-i$, with $\beta_j\neq\alpha_{j+2}$, and $\beta_j>\alpha_{j+2i}$ for the maximal such $j$. 

Define the set $\mathcal{C}_{a,i}$ to be the complement of $\mathcal{D}_{a,i}$ in $\Tab^{(2)}_{(a+i,a-i)}$. Thus, $T\in\mathcal{C}_{a,i}$ if either $\beta_j=\alpha_{j+2i}$ for all $j=1,\ldots,a-i$, or $\beta_j<\alpha_{j+2i}$ for the maximum $j$ where $\beta_j\neq\alpha_{j+2i}$.
\end{definition}

\begin{lemma}\label{lem:property blah bijection}
    Let $a\geq 1$ and $0\leq i\leq a-1$. Then there is a weight-preserving bijection between the sets
    \[\mathcal{D}_{a,i}\simeq \mathcal{C}_{a,i+1}.\]
\end{lemma}

\begin{proof}
    We first give a weight-preserving map $\Phi:\mathcal{D}_{a,i}\to\mathcal{C}_{a,i+1}$. Let $T\in\mathcal{D}_{a,i}$, with $j$ the maximal index where $\beta_j\neq\alpha_{j+2i}$ (so $\beta_j>\alpha_{j+2i}$). If $j=a-i$, then move the box corresponding to $\beta_j$ to the end of the first row. Else, swap the boxes
    \[\underbrace{\alpha_{j+2i+1},\ldots,\alpha_{a+i}}_{\text{$a-i-j$ boxes}}\leftrightarrow \underbrace{\beta_{j},\ldots,\beta_{a-i}}_{\text{$a-i-j+1$ boxes}}.\]
    Then $\Phi(T)\in\mathcal{C}_{a,i+1}$.

    We now give an inverse map $\Psi:\mathcal{C}_{a,i+1}\to\mathcal{D}_{a,i}$. Let $T\in\mathcal{C}_{a,i+1}$. If $\beta_j=\alpha_{j+2i}$ for $j=1,\ldots,a-i$, then swap the boxes
    \[\underbrace{\alpha_{2i},\ldots,\alpha_{a+i}}_{\text{$a-i+1$ boxes}}\leftrightarrow\underbrace{\beta_1,\ldots,\beta_{a-i}}_{\text{$a-i$ boxes}}.\]
    Else, let $j$ be the maximal index where $\beta_j\neq\alpha_{j+2i}$ (so $\beta_j<\alpha_{j+2i}$), and swap the boxes
    \[\underbrace{\alpha_{j+2i},\ldots,\alpha_{a+i}}_{\text{$a-i-j+1$ boxes}}\leftrightarrow\underbrace{\beta_{j+1},\ldots,\beta_{a-i}}_{\text{$a-i-j$ boxes}}.\]
    Then $\Psi(T)\in\mathcal{C}_{a,i}$. One can check that $\Phi,\Psi$ are inverse to each other.
\end{proof}

\begin{proposition}\label{prop:a=b=d schur poly}
    For $a\geq 1$, let $\mathcal{D}$ denote the set of $2$-semistandard tableaux of size $(a,a)$ that do not have identical rows. Then
    \[\sum_{j=1}^a (-1)^{j-1}s^{(2)}_{(a+j,a-j)}=\sum_{T\in\mathcal{D}} t^T.\]
\end{proposition}

\begin{proof}
    Observe that $\mathcal{D}$ is exactly equal to $\mathcal{D}_{a,0}$, so $\sum_{T\in\mathcal{D}}t^T=\sum_{T\in\mathcal{C}_{a,1}}t^T$ by Lemma~\ref{lem:property blah bijection}. For any $0\leq i\leq a$, we have
    \[s^{(2)}_{(a+i,a-i)}=\sum_{T\in\mathcal{D}_{a,i}}t^T + \sum_{T\in\mathcal{C}_{a,i}}t^T.\]
    In addition,
    \[s^{(2)}_{2a,0}=\sum_{T\in\mathcal{C}_{a,a}}t^T.\]
    Combining these statements with Lemma~\ref{lem:property blah bijection} gives the desired result.
\end{proof}

\begin{theorem}\label{thm:main}
    Let $a\geq b\geq d$. Then,
    \[\left[\left(\overline{I}^d/\overline{I}^{d+1}\right)_{(a,b)}\right]=\begin{cases}
    \sum_{j=1}^a(-1)^{j-1}s^{(2)}_{(a+j,a-j)}&\quad\text{if $a=b=d$},\\ \\
    s^{(2)}_{(a,a)}+\sum_{j=2}^a (-1)^j s^{(2)}_{(a+j,a-j)}&\quad\text{if $a-1=b-1=d$},\\ \\
    s^{(2)}_{(a+b-d,d)}&\quad\text{otherwise.}
    \end{cases}\]
\end{theorem}

\begin{proof}
    For any $a\geq b$, we have that $[\overline{R}_{(a,b)}]=\sum_{d=0}^b s^{(2)}_{(a+b-d,d)}$.
    By Proposition~\ref{prop:schur poly tab corr}, this is equal to
    \begin{align*}
        \sum_{d=0}^b\left(\sum_{T\in\Tab^{(2)}_{(a+b-d,d)}}t^T\right).
    \end{align*}
    By Theorem~\ref{thm:main span}, $\mathcal{B}_{a,b}=\bigcup_{d=0}^b\mathcal{B}_{a,b,d}$ spans $\overline{R}_{(a,b)}$. Then,
    \[\dim_{\bk}(\overline{R}_{(a,b)})\leq \#\mathcal{B}_{a,b}\leq \#\left(\bigcup_{d=0}^b\Tab^{(2)}_{(a+b-d,d)}\right)=\dim_{\bk}(\overline{R}_{(a,b)}),\]
    and so we have that $\mathcal{B}_{a,b}$ is indeed a $\bk$-basis for $\overline{R}_{(a,b)}$. By construction, $\mathcal{B}_{a,b,d}\subset\overline{I}^d_{(a,b)}$. Thus, by downwards induction on $d$, we have that the residue classes of elements of $\mathcal{B}_{a,b,d}$ are a $\bk$-basis for $(\overline{I}^d/\overline{I}^{d+1})_{(a,b)}$.
    For any $2$-semistandard tableau $T$, the basis element $G_{T,a}$ is a weight vector with $t^{\text{wt}(G_{T,a})}=t^T$. Then, if $a>b$ or $a-2=b-2\geq d$, we have
    \[\left[\left(\overline{I}^d/\overline{I}^{d+1}\right)_{(a,b)}\right]=\sum_{G_{T,a}\in\mathcal{B}_{a,b,d}}t^{\text{wt}(G_{T,a})}=\sum_{T\in\Tab^{(2)}_{(a+b-d,d)}}t^T=s^{(2)}_{(a+b-d,d)}.\]
    By Proposition~\ref{prop:a=b=d schur poly}, we have the desired character formula for the $a=b=d$ case. Finally, for the $a-1=b-1=d$ case, we see that the character is given by
    \begin{align*}
        \sum_{G_{T,a}\in\mathcal{B}_{a,a,a-1}}t^{\text{wt}(G_{T,a})}&=\sum_{T\in\Tab^{(2)}_{(a+1,a-1)}}t^T+\sum_{T\in\left(\Tab^{(2)}_{(a,a)}\setminus\mathcal{D}\right)}t^T\\
        &=s^{(2)}_{(a+1,a-1)}+s^{(2)}_{(a,a)}-\left(\sum_{j=1}^a (-1)^{j-1}s^{(2)}_{(a+j,a-j)}\right)\\
        &=s^{(2)}_{(a,a)}+\left(\sum_{j=2}^a (-1)^j s^{(2)}_{(a+j,a-j)}\right).
        \end{align*}
\end{proof}

\subsection{Simple modular representations}\label{subsec:irreps dir} 

A direction of future work is to try to find the simple composition factors of $\overline{R}_{(a,b)}$ in the $\char(\bk)=2$ case. For $a>b$, the following proof shows that if $L(\mu)$ is a simple composition factor of $\overline{R}_{(a,b)}$, then it must be $2$-restricted, i.e., $\mu$ is a partition consisting of two columns with different lengths or is a single column.

\begin{proof}[Proof of Conjecture~\ref{conj:grv simples} for $p=2$]
    In characteristic 2, we have that $T_p\,\Sym^a(-)$ is the Schur functor $\bfS_{(1^a)}(-)$. Then by Pieri's rule, $T_p\,\Sym^a(\bk^n)\otimes T_p\,\Sym^b(\bk^n)$ has a filtration with graded pieces $\bfS_\lambda(\bk^n)$, where $\lambda$ is a partition with at most two columns. In particular, every $\lambda$ appearing in the filtration has the form $(a+i,b-i)^\dagger$, where $0\leq i\leq b$. Since $a-b\geq 1$, each $\lambda$ appearing can never have two columns of the same length.

    For a given $\bfS_\lambda(\bk^n)$ appearing in the filtration, any simple composition factor $L(\mu)$ of $\bfS_\lambda(\bk^n)$ must be such that $\mu\leq\lambda$ under the dominance order, as $\lambda$ is the highest weight occurring in $\bfS_\lambda(\bk^n)$. This implies that $\mu$ must also have at most two columns, and they cannot have the same length. Therefore, any $\mu$ for which $L(\mu)$ appears as a composition factor of $T_p\,\Sym^a(\bk^n)\otimes T_p\,\Sym^b(\bk^n)$ must be $2$-restricted.
\end{proof}

By \cite[Theorem 4.4]{wal94}, $s^{(2)}_{(a,b)}$ is the character of the simple $\GL_n$-representation $L((a,b)^\dagger)$ if $(a,b)$ is \emph{Carter}, meaning the highest power of $2$ dividing the hook-length is constant down each column. Walker further showed that if $a>b$ and $(a,b)$ is not Carter, then $s^{(2)}_{(a,b)}$ is not the character of a simple $\GL_n$-representation for $n\gg 0$. Thus, for many $a\geq b\geq d$, one knows that $(\overline{I}^d/\overline{I}^{d+1})_{(a,b)}$ is not simple.

\subsection{Exterior algebras}\label{subsec:exterior}

Cauchy's identity states that the exterior algebra $\bigwedge(\bk^2\boxtimes\bk^n)$ has a filtration by $(\GL_2\times\GL_n)$-representations $\bfS_\lambda(\bk^2)\boxtimes\bfS_{\lambda^\dagger}(\bk^n)$. This identity holds for any $\char(\bk)$, and it gives a decomposition into simple representations if $\char(\bk)=0$.

In this paper, the ring $\overline{R}$ is equal to the exterior algebra since $\char(\bk)=2$. The equality of symmetric polynomials $s^{(2)}_{(a,b)}=s_{(a,b)^\dagger}$ shows that for $a>b$, the characters of our filtration of $\overline{R}_{(a,b)}$ agrees with those of the filtration given by Cauchy's identity. 

It would be interesting to try to extend this line of work to any $\char(\bk)$. Developing standard monomial theory for exterior algebras could lead to interesting algebraic and geometric results, as in the case for classical standard monomial theory and symmetric algebras.

\bibliographystyle{alpha}
\bibliography{ref}

\end{document}